\documentclass[reqno]{amsart}
\usepackage[foot]{amsaddr}
\usepackage[utf8]{inputenc}
\usepackage{amsthm}
\usepackage{enumerate}
\usepackage{amsfonts}
\usepackage{amssymb}
\usepackage{amsmath}
\usepackage{mathabx}
\usepackage{color}
\usepackage{wasysym}
\usepackage{tikz-cd}
\usepackage{hyperref}
\usepackage[normalem]{ulem}

\newtheorem{theorem}{Theorem}
\newtheorem{corollary}[theorem]{Corollary}
\newtheorem{lemma}[theorem]{Lemma}
\newtheorem{proposition}[theorem]{Proposition}
\theoremstyle{definition}
\newtheorem{definition}[theorem]{Definition}
\newtheorem{example}[theorem]{Example}
\newtheorem{remark}[theorem]{Remark}

\title[Generalizing shellable complexes and matroids]{A lattice framework for generalizing shellable complexes and matroids}
\author{Rakhi Pratihar$^1$}
\address[1]{Inria Centre de Saclay, Campus Polytechnique, Palaiseau, France}
\author{Tovohery H. Randrianarisoa$^2$ and Klara Stokes$^2$}
\address[2]{Department of Mathematics and Mathematical Statistics, Ume\aa ~University, Ume\aa, Sweden}
\email{rakhi.pratihar@inria.fr, klara.stokes@umu.se, tovo@aims.ac.za}

\newcommand{\N}{\mathbb{N}}

\renewcommand{\precdot}{\prec\mathrel{\mkern-5mu}\mathrel{\cdot}}
\newcommand{\covered}{\precdot}
\newcommand{\0}{\hat{0}}
\newcommand{\1}{\hat{1}}
\newcommand{\<}{\left<}
\renewcommand{\>}{\right>}
\newcommand{\Fq}{\mathbb{F}_q}
\newcommand{\FF}{\mathbb{F}}
\newcommand{\power}[1]{\mathrm{pow}( #1 )}

\newcommand{\I}{\mathcal{I}}
\newcommand{\B}{\mathcal{B}}

\newcommand{\A}{\mathfrak{A}}

\newcommand{\depth}{\mathrm{depth}}

\newcommand{\PP}{\mathcal{P}}

\DeclareSymbolFont{slenderlargesymbols}{OMX}{ccex}{m}{n}
\DeclareMathSymbol{\prod}{\mathop}{slenderlargesymbols}{"51}
\begin{document}

\begin{abstract}
    We introduce the notion of \emph{power lattices} that unifies and extends the equicardinal geometric lattices, Cartesian products of subspace lattices, and multiset subset lattices, among several others. The notions of shellability for simplicial complexes, $q$-complexes, and multicomplexes are then unified and extended to that of complexes in power lattices, which we name as \emph{$\PP$-complexes}. A nontrivial class of shellable $\PP$-complexes are obtained via $\PP$-complexes of the independent sets of a matroid in power lattice, which we introduce to generalize matroids in Boolean lattices, $q$-matroids in subspace lattices, and sum-matroids in Cartesian products of subspace lattices. We also prove that shellable $\PP$-complexes in a power lattice yield shellable order complexes, extending the celebrated result of shellability of order complexes of (equicardinal) geometric lattices by Bj{\"o}rner and also, a recent result on shellability of order complexes of lexicographically shellable $q$-complexes. Finally, we provide a construction of matroids on the lattice of multiset subsets from weighted graphs. We also consider a variation of Stanley-Reisner rings associated with shellable multicomplexes than the one considered by Herzog and Popescu and proved that these rings are sequentially Cohen-Macaulay.
\end{abstract}

\maketitle
\section{Introduction}

Shellability is an important notion having interesting applications in combinatorics, commutative algebra, and algebraic topology. Historically, it was introduced to prove the Euler-Poincar{\'e} formula for higher-dimensional convex polytopes. The shellability property was first implicitly assumed by Schl{\"a}fli \cite{Sch} in 1852 and was formally established after many decades by Bruggesser and Mani \cite{BM} in 1971.

The original definition of shellability for polyhedral complexes takes the following form for abstract simplicial complexes pioneered by Bj{\"o}rner and Wachs  \cite{Bjo92}. An abstract simplicial complex $\Delta$ is said to
be shellable if it is pure (i.e., all its facets or maximal faces have the same dimension) and there is a linear ordering $F_1, \ldots, F_t$ of its facets such that for each $j = 2, \ldots, t$, the complex
$\<F_j \> \cap \< F_1, \ldots, F_{j-1} \>$ is generated by a nonempty set of maximal proper faces of $F_j$. Here for $i = 1, \ldots, t$, by $\< F_1, \ldots, F_i \>$ we denote the complex generated by
$F_1, \ldots, F_i$, i.e., the smallest simplicial complex containing $F_1, \ldots, F_i$.

 From the topological point of view, a shellable simplicial complex is weak homotopy equivalent to a wedge sum of spheres, thus the homology groups are well understood \cite[\S2.3]{Hatcher}. Shellable simplicial complexes are important in commutative algebra, partly because their ``face rings'' or Stanley-Reisner rings (over any field) are Cohen-Macaulay. These rings were introduced independently by Hochster and Stanley and they have nice properties. Since Gr\"obner degenerations of the coordinate rings of many classes of algebraic varieties turn out to be Stanley-Reisner rings of simplicial complexes, shellability provides a tool to prove their Cohen-Macaulayness. Even if the shellable simplicial complexes are not pure, the associated Stanley-Reisner ring is still sequentially Cohen-Macaulay \cite[Chapter 2, Section 2]{Stan}. The notion of shellability has been generalized to multicomplexes (e.g. \cite[Definition 3.4]{Cim06}). It was also shown that when a pure multicomplex is shellable, the associated topological space is a wedge of spheres \cite[Corollary 3.10]{Cim06}. More recently, $q$-analogues of these results are obtained in \cite{GPR22, GPTVW24}where (multi)complexes are replaced with $q$-complexes, introduced at least by Rota \cite{Rota71} and then by Alder \cite{Alder10}. 
It was shown in \cite{GPTVW24} that when a $q$-complex is shellable, the associated topological space is also a wedge of spheres.

 Motivated by Stanley's work on R-labeling, Bj{\"o}rner introduced the notions of EL-labelling and CL-labelling that give efficient methods for establishing shellability.  Important known classes of shellable simplicial complexes include the boundary complex of a convex polytope \cite{BM}, the order complex of a bounded, locally upper semimodular poset \cite{folkman}, and matroid complexes, i.e., complexes formed by the independent subsets of matroids \cite{Porvan}. Among different constructions, one can obtain shellable multicomplexes and $q$-complexes as independent sets of discrete polymatroids \cite{HH02} and $q$-matroids \cite{GPR22}, respectively. For comprehensive background and the proofs supporting these claims, we refer to the literature, including the books by Stanley \cite{Stan} and Bruns and Herzog \cite{BH}, the compilation of lecture notes in \cite{GSSV}, and Björner's survey article \cite{Bjo92}.


Furthermore, a more general notion of sum-matroids has been introduced in \cite{PPR23}, which is a direct generalization of the notions of matroids and $q$-matroids. Therefore, it is natural to ask if the corresponding ``independent elements'' have a ``shellable'' property and if the associated topological space is also a wedge of spheres. With this primary motivation, we introduce the notion of \emph{power lattices} and give an affirmative answer to this question. Within this lattice, we define a general notion of ``$\PP$-complexes'' and their shellability. We will show that the associated topological space is again a wedge of spheres. Our approach is similar to the one in \cite{GPTVW24} where we show that the order complex associated to a shellable $\PP$-complex is a shellable simplicial complexes, and then we use the homotopy equivalence between the associated topological spaces. 
We obtain non-trivial classes of shellable $\PP$-complexes by considering the $\PP$-complexes of independent sets of matroids in power lattices, that we introduce to generalize the notion of matroids, polymatroids, $q$-matroids. 

In the remaining part of the paper, we focus on the Stanley-Reisner ring associated with multicomplexes.
 It was shown in \cite[Corollary 10.6]{HP06} that the corresponding quotient ring, which we denote by $R(\Gamma)$ is (sequentially) Cohen-Macaulay when $\Gamma$ is a shellable multicomplex. However, as mentioned in \cite[Remark 1.5]{Cim06}, the ideal $I(\Gamma)$ is not the same as the Stanley-Reisner ideal $I_\Gamma$ when the multicomplex is a simplicial complex. In \cite{HP06}, they consider the ideal $I(\Gamma)$ associated with a multicomplex $\Gamma$ where $I(\Gamma)$  is the monomial ideals spanned by all monomials which are not in $\Gamma$. 
To provide an appropriate generalization, we consider a monomial ideal $I_\Gamma$ generated by monomials not in $\Gamma$ but belonging to a power lattice of multiset subsets of some fixed multiset. We justify this choice by showing that $I_\Gamma$ can be obtained by using the same construction of the Stanley-Reisner ring using a section ring of sheaves in \cite{Yuz87}. Our result, therefore, states that the quotient ring $R_\Gamma$ by the ideal $I_\Gamma$ is sequentially Cohen-Macaulay whenever $\Gamma$ is a shellable multicomplex. 
Since, in general, a shellable multicomplex can be obtained from the independent elements of a matroid on the lattice of multiset subsets, we show how we can construct such matroids from weighted graphs. This construction can be seen as a generalization of graphic matroids.
    
The rest of this paper is organized as follows. In Section \ref{sec:powerlattices}, we introduce the notion of \emph{power lattices} and study their basic properties.
We describe some classes of lattices that it generalizes including the \emph{equicardinal geometric lattices} and the \emph{lattices of multiset subsets}. In Section \ref{sec:complexesinpowerlattices}, we introduce the notion of complexes in power lattices that generalizes the notions of simplicial complexes, $q$-complexes and multicomplexes. In this section, we also introduce the notion of shellability for complexes in a power lattice that extends and generalizes the existing notions of shellability for the aforementioned complexes. The main result of this section is the shellability of order complexes of a shellable complex in a power lattice. Then, we introduce the notion of matroids in power lattices in Section \ref{sec:matroidsinpowerlattices} and prove that the complexes induced by them are shellable complexes in power lattices. We construct a matroid in the power lattice of multiset subsets from weighted graphs. Finally, in Section \ref{sec:stanleyreisnerrings}, we associate Stanley-Reisner rings to multicomplexes, and we establish the sequentially Cohen-Macaulayness of the Stanley-Reisner rings of shellable multicomplexes.



\section{Power lattices}\label{sec:powerlattices}

In this section, we start with recalling basic notions and definitions regarding finite posets. Then we introduce power lattices that generalize and extend the equicardinal geometric lattices, subset lattices, subspace lattices and lattices of multiset subsets, among others. We give a total order on the elements of the same rank of a power lattice. This total order will be useful to study the complexes associated to the lattice in the subsequent sections.
\begin{definition}[Posets]
A partially ordered set or poset is a set $P$ with a binary relation $\preceq$ defined on its elements, which satisfies 
\begin{enumerate}[P1:]
    \item (Reflexive) for all $x \in P$, $x \preceq x$,

    \item (Transitive) if $x \preceq y$ and $y \preceq z$, then $x \preceq z$,

    \item (Antisymmetric) if $x \preceq y$ and $y \preceq x$, then $x = y$.
\end{enumerate}
Throughout the text, we will use $(P, \preceq)$ to denote a poset. If the relation $\preceq$ is clear from the context, we simply use $P$ omitting the symbol $\preceq$.
\end{definition}

The symbol $\preceq$ reads ``precedes" or ``contained in" or ``is less than or equal to". For $x,y\in P$, if $x\preceq y$ and $x\neq y$, then one writes $x \prec y$. If $x\prec y$ and there exists no $z\in P$ such that $x\prec z\prec y$, then one writes $x\covered y$ that reads $y$ covers $x$ (or $x$ is covered by $y$). A poset $P$ is said to have a lower (resp. upper) bound $x\in P$ if $\forall y\in P$, $x\preceq y$ (resp. $y\preceq x$). A poset is \emph{bounded} if it admits both an upper and a lower bound (which must be unique). In a bounded poset, the unique lower (resp. upper) bound is denoted by $\0$ (resp. $\1$). A \emph{subposet} $(S,\preceq)$ of $(P,\preceq)$ is a subset $S\subset P$ with the partial order in $S$ induced from the partial order $\preceq$ on $P$.
 

\begin{definition}
Let $x,y$ be elements of the poset $(P,\preceq)$. The \emph{join} of $x$ and $y$ (if it exists) is an element $z\in P$ such that $x, y \preceq z$ and every $u\in P$ with $x, \, y \preceq u$ satisfies $z\preceq u$. The join of $x$ and $y$ is denoted by $x\vee y$. The \emph{meet} of $x$ and $y$ (if it exists) is an element $z\in P$ such that $z\preceq x, \, y$ and every $u\in P$ with $u\preceq x, \,y$  satisfies $u\preceq z$. The meet of $x$ and $y$ is denoted by $x\wedge y$.

    
\end{definition}
By definition, meet and join of two elements in a poset are unique.
The meet and join satisfy the following properties.
\begin{equation}\label{eq:5}
x\preceq y \Longleftrightarrow x\vee y = y \Longleftrightarrow x\wedge y = x.    
\end{equation}

\begin{definition}[Lattices]
 A \emph{join-semilattice} (resp. \emph{meet-semilattice}) is a poset $P$ in which every pair of elements admits a join (resp. meet). A poset is called a \emph{lattice} if it is both a join-semilattice and a meet-semilattice.
\end{definition}

\begin{definition}[Ranked lattices]\label{def:rankedlattice}
A rank function on a poset $(P, \preceq)$ is a non-negative integer valued function $\rho: P\rightarrow \N$ which satisfies 
\begin{enumerate}[(i)]
    \item if $x \prec y\in P$, then $ \rho(x)<\rho(y)$,
    \item if $x\covered y$, then $\rho(y)=\rho(x)+1$.
\end{enumerate}
A lattice $(P, \preceq)$ with a rank function $\rho$ is called a \emph{graded} or \emph{ranked} lattice, often denoted as $(P, \preceq, \rho)$ or simply as $(P, \rho)$. We denote by $P(l)$ the elements of $P$ of rank $l$.
\end{definition}

For a bounded ranked lattice $P$, the standard convention is to take $\rho(\0)=0$ and the rank of the lattice $P$ is defined to be $\rho(P):=\rho(\1)$. 

\begin{definition}[Semimodular lattice]
A semimodular lattice is a ranked lattice $(P, \rho)$ such that for all $x,y\in P$, 
\begin{equation}\label{eq:semi}
    \rho(x\vee y)+\rho(x\wedge y)\leq \rho(x)+\rho(y).
\end{equation}
If equality holds in \eqref{eq:semi}, the lattice is called modular.
\end{definition}
The elements of $P(1)$ of a ranked lattice $(P, \rho)$ are called atoms of $P$. These elements cover $\0$. We introduce the notion of multiplicity of an element of $P$ with respect to an atom.
    

\begin{definition}[Multiplicity in a lattice]
Let $(P, \rho)$ be a ranked lattice. For any element $x\in P$, we describe the set of all atoms preceding $x$ by
$${A}(x):=\{w\in P(1)\colon w\preceq x\}.$$
In the case $A(x)=\{w\}$, we call $x\in P$ a \emph{power} of the atom $w$. The set of all powers of an atom $w$ is denoted by $\power{w}:=\{ y\in P\colon A(y) = w\}$. In this case $v_w(y):=\rho(y)$ is called the \emph{$w$-multiplicity} of $y$ (or the valuation of $y$ at $w$). 
    
    For an atom $w\preceq y$, the \emph{$w$-multiplicity} of $y$ (or the valuation of $y$ at $w$) is defined as
    \[
     v_w(y):=\max\{ v_w(z)\colon z\preceq y\text{ and } z\in \power{w} \}.
    \]
    If an atom $w\npreceq y$, then we define $v_w(y)=0$.
\end{definition}

\begin{definition}\label{def:powerlattice}
A power lattice is a finite, semimodular lattice $(P, \rho)$ which satisfies the following properties:
\begin{enumerate}[{\em (i)}]
    \item For any $w \in P(1)$ and integer $r$, there exists at most one element in $\power{w}$ of rank $r$.
    \item\label{four} For any $x,y \in P$, $\sum\limits_{w\in P(1)}v_w(x) = \sum\limits_{w\in P(1)}v_w(y)\Longleftrightarrow\rho(x) = \rho(y)$.
\end{enumerate}  
\end{definition}

Since by definition, a power lattice is finite, it is bounded. If $y$ is a power of $x$, for some atom $x$, then we write $y = x^{\rho(y)} = x^{v_x(y)}$. 

\begin{lemma}\label{lem:iniqualityvaluation}
Given a power lattice $P$,
    if $x\preceq y$, then $v_w(x)\leq v_w(y)$ for all $w \in P(1)$. 
\end{lemma}
\begin{proof}
    Suppose $x\preceq y$, then for an atom $w$, $w^{v_w(x)}\preceq x\preceq y$ and therefore $v_w(x)\leq v_w(y)$. 
\end{proof}
Later we will see that the converse of this lemma is also true.

\begin{example}\

\begin{minipage}{0.5\textwidth}
    The example on the right provides a finite ranked lattice $(P,\rho)$ which is semimodular. Moreover, $3$ is the only element of rank 2 which is power of the atom $6$. Notice that $\rho(3) = \rho(2) = 2$, but $A(3) = \{6\}$ and $A(2) = \{4, 5,6 \}$. Thus the total valuation of 3 is $\sum\limits_{w\in P(1)} v_w(3) = 2$, whereas $\sum\limits v_w(2) = 3$. Hence it does not satisfy Definition \ref{def:powerlattice} $(ii)$ and thus $P$ is not a power lattice.
\end{minipage}
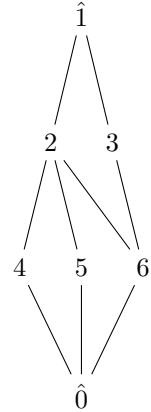
\hfill
\begin{minipage}{0.5\textwidth}
        \centering
        \begin{tikzpicture}
    \matrix (A) [matrix of nodes, row sep=1.2cm]
    { 
        &&$\1$&&\\  
	& $2$ & & $3$\\
	$4$ & & $5$ & & $6$\\
	&&$\0$&&\\
    };
    \draw (A-1-3)--(A-2-2);
    \draw (A-1-3)--(A-2-4);
    \draw (A-2-2)--(A-3-1);
    \draw (A-2-2)--(A-3-3);
    \draw (A-2-2)--(A-3-5);
    \draw (A-2-4)--(A-3-5);
    \draw (A-3-1)--(A-4-3);
    \draw (A-3-3)--(A-4-3);
    \draw (A-3-5)--(A-4-3);
\end{tikzpicture}
\end{minipage}
\end{example}

Following the non-example, now we give a list of some lattices which are included in the class of power lattices.
\begin{example}\label{exa:1}

\begin{enumerate}[(1)]
    \item For a finite set $A$, the Boolean lattice $2^A$ of all subsets of $A$ is a power lattice $(2^A,\subseteq)$, where $\vee$ is the union of two sets and $\wedge$ is the intersection. The rank $\rho$ is given by the size of subsets. For any subset $U\in 2^A$, $v_{\{x\}}(U)=1$ if $x\in U$ and $v_{\{x\}}(U)=0$, otherwise. Thus for any atom $w$, the power set of $w$ is $\power{w} = \{ w\}$. Also, the condition of equal rank elements having equal cardinality in Definition \ref{def:powerlattice} $(ii)$ is satisfied from the definition of the rank function.

    \item Let $(\Sigma(\Fq^n), \subseteq)$ be the lattice of $\Fq$-subspaces of the $n$-dimensional vector space $\Fq^n$ over $\Fq$, where $\vee$ is the sum of two spaces and $\wedge$ is the intersection. Then $(\Sigma(\Fq^n),\subseteq)$ is power lattice with the rank function $\rho$ given by dimension. Here the atoms are the one-dimensional subspaces. For $U\in \Sigma(\Fq^n)$, $v_{\<x\>}(U)=1$, if $x\in U$ and $v_{\<x\>}(U)=0$, otherwise. The conditions $(i)$ and $(ii)$ of Definition \ref{def:powerlattice} are satisfied for similar reasons as in the previous example.

    \item Let $n$ be an integer. The set $div(n)$ of all divisors of $n$ forms a power lattice $(div(n),|)$ where $a|b$ if a divides $b$, $\vee$ is the least common multiple and $\wedge$ is  greatest common divisor of two numbers. The rank function $\rho$ is the number of prime factors counted with multiplicity. If $x\in div(n)$ and $x=p_1^{a_1}\dots p_t^{a_t}$ is the prime factorization of $x$, then $v_{p_i}(x)=a_i$, if $p$ doesn't divide $x$, then $v_{p}(x)=0$. This is ``equivalent'' to the lattice of multiset we define next and which we will also see in a later section.
    \item Let $M$ be a multiset. The set $2^M$ of all multiset subsets of $M$ forms a power lattice $(2^M,\subseteq)$ where $\vee$ is the multiset union of two sets and $\wedge$ is the intersection. The rank function $\rho$ is defined by the size of subsets (with multiplicity). Let $U\in 2^M$, if $x\in U$, then $v_{\{x\}}(U)$ is the multiplicity of $x$ in $U$ and if $x\notin U$, then $v_{\{x\}}(U)=0$. 
    \item For $i=1,\dots,l$, let $n_i$ be an integer. The Cartesian product of lattices $\Sigma(\Fq^{n_1})\times \dots\times \Sigma(\Fq^{n_l})$ is a power lattice. $\vee$ is defined by the componentwise sums of subspaces whereas $\wedge$ is defined by the componentwise intersections. The rank function $\rho$ is the sum of the dimension of the components. Let $U=(U_1,\dots,U_l)$ and let $u=(\langle 0\rangle,\dots,\langle 0\rangle,\langle x\rangle, \langle 0\rangle, \dots,\langle 0\rangle)$ where $x\neq 0$ and $\<x\>$ is at the position $i$  for some $1\leq i\leq l$. If  $x\in U_i$ then $v_{u}(U)=1$. Otherwise, $v_{u}(U)=0$.
    \item  If $G$ is a cyclic group, then the lattice of subgroups of $G$ forms a power lattice $L(G)$ where $\vee$ is given by the subgroup generated by the union of two subgroups and $\wedge$ is the intersection. Let $H\in L(G)$ such that $|H| = \prod p_i^{l_i}$, where $p_i$'s are pairwise coprime. Then $\rho(H)=\sum_i l_i$. If $C\in L(H)$ with $|C|=p_i$, then $v_{C}(H)=l_i$, otherwise if $C\notin L(H)$, then $v_{C}(H)=0$.
    \item  Let $G$ be a finite abelian group, not necessarily cyclic. By the fundamental theorem of finite abelian groups, $G$ can be expressed as the direct sum of cyclic groups of prime-power order say $G=\prod_{i=1}^s G_i$. So the lattice of subgroups of $G$ can be defined as the Cartesian product of lattices $L(G) = \prod_{i=1}^s L(G_i)$. Similarly to the Cartesian product of lattice of subspaces, we also get a power  lattice of subgroups of abelian groups.

    \item
    Equicardinal geometric lattices are power lattices. Indeed, a geometric lattice is equicardinal if its elements of equal rank have equal cardinality. Note that, the atomistic property of a geometric lattice implies that the valuation of an element $x$ w.r.t. any atom $w$ is $v_w(x) = 1$ if $w \in A(x)$ and 0, otherwise. Thus the property $(ii)$ of Definition \ref{def:powerlattice} becomes equivalent to the property of the same rank elements having the same cardinality. The other properties follows from the definition of geometric lattices.


\end{enumerate}
\end{example}

\begin{remark}
\begin{enumerate}[(a)]
    \item Following the characterization of geometric lattices as lattices of flats of simple matroids, the class of equicardinal geometric (semi)lattices include the equicardinal matroids. For more on this topic, one can refer to, e.g., the article \cite{ABS91} by Alon, Babai, and Suzuki and the references mentioned there.

    \item  
    { There exists a solvable non-abelian group $G$ whose subgroup lattice do not satisfy the property (i) in the definition of a power lattice. For example, take the quaternion group $Q_8=\{\pm 1, \pm i, \pm j,\pm k\}$. Then $\{-1,1,i,-i\}$ and $\{-1,1,j,-j\}$ are two powers of the atom $\{-1,1\}$ and they have the same order.}

\end{enumerate}
    
\end{remark}

\begin{lemma}\label{lem:valuation}
    Let $P$ be a ranked lattice with $p \in P(1)$. Then for $x, \, y \in P$, 
    \begin{enumerate}[(1)]
        \item $v_p(x\wedge y)= \min\{v_p(x),v_p(y)\}$,
        \item $v_p(x\vee y)\geq \max\{v_p(x),v_p(y)\}$.
    \end{enumerate}
\end{lemma}
\begin{proof}
    Since $x\wedge y\preceq x$ and $x\wedge y\preceq y$, we have $v_p(x\wedge y)\leq v_p(x)$ and $v_p(x\wedge y)\leq v_p(y)$. Therefore $v_p(x\wedge y)\leq \min\{v_p(x),v_p(y)\}$. Now, $p^{\min\{v_p(x),v_p(y)\}}\preceq x$ and $p^{\min\{v_p(x),v_p(y)\}}\preceq y$. Therefore $p^{\min\{v_p(x),v_p(y)\}}\preceq x\wedge y$. Hence $\min\{v_p(x),v_p(y)\}\leq v_p(x\wedge y)$. Thus $\min\{v_p(x),v_p(y)\}= v_p(x\wedge y)$.
    For the second point, $p^{v_p(x\vee y)}\preceq x$ and $p^{v_p(x\vee y)}\preceq x$. Hence $v_p(x\vee y)\geq v_p(x)$ and $v_p(x\vee y)\geq v_p(x)$.
\end{proof}
\begin{remark}
    Notice that in the second point of Lemma \ref{lem:valuation}, strict inequality can happen. For example if we have the lattice of subspaces, the sum of two subspaces can produce a new vector which was not in the two original subspaces.
\end{remark}


The rest of the section deals with defining a total order on the elements of same rank of a power lattice. For that, first we will introduce the notion of factorization of an element of $P$ into atoms. We consider an arbitrary but fixed total order $\preceq_1$ on $P(1)$ i.e., the atoms of $P$. 
\begin{definition}[Factorization]
Let $x\in P$ and $A(x)=\{x_1,\dots,x_l\}$ such that $x_1 \prec_1 x_{2} \prec_1 \cdots \prec_1 x_l$. Then the factorization of $x$ is defined as
\[
F(x)= (\overbrace{x_1,\dots,x_1}^{v_{x_1}(x)\text{ times }}, \dots, \overbrace{x_l,\dots,x_l}^{v_{x_l}(x)\text{ times }}).
\]
The length of the factorization of $x$ is defined as $\sum\limits_{x_i \in A(x)} v_{x_i}(x)$. For brevity, we will denote $F(x)$ as
\[
F(x) = \prod_{i=1}^l x_i^{v_{x_i}(x)}.
\]
\end{definition}
 
\begin{remark}\label{rem:1}
Because of the property \eqref{four} in Definition \ref{def:powerlattice} of a power lattice, the factorizations of two elements of the same rank must have the same length i.e., the cardinalities are the same when counted with multiplicity.
\end{remark}

We show next that factorization uniquely defines an element of a power lattice.
\begin{lemma}\label{lemma1}
Let $(P, \preceq)$ be a power lattice. For $x \in P$, if $F(x) = \prod_{i=1}^l x_i^{v_i}$, then $x = \bigvee\limits_{i=1}^l x_i^{v_i}$. Moreover, if $x,y \in P$ such that $x\neq y$, then $F(x)\neq F(y)$.
\end{lemma}
\begin{proof}
Suppose that $F(x) = \prod_{i=1}^l x_i^{v_i}$. We know that $v_i=v_{x_i}(x)$ and from the definition of $v_{x_i}(x)$, $x_i^{v_i} \preceq x$ and therefore, by definition of join, $\bigvee\limits_{i=1}^l x_i^{v_i} \preceq x$. By Lemma \ref{lem:iniqualityvaluation}, $v_{x_i}\left(\bigvee\limits_{i=1}^l x_i^{v_i}\right)\leq v_{x_i}(x)$. On the other side, since $x_i^{v_i} \preceq \bigvee\limits_{i=1}^l x_i^{v_i}$ and $v_i=v_{x_i}(x)$, by Lemma \ref{lem:iniqualityvaluation}, $v_{x_i}(x) \leq v_{x_i}\left(\bigvee\limits_{i=1}^l x_i^{v_i}\right)$. Hence $v_{x_i}(x) = v_{x_i}\left(\bigvee\limits_{i=1}^l x_i^{v_i}\right)$. It is also clear that $A(x)\subseteq A\left(\bigvee\limits_{i=1}^l x_i^{v_i}\right)$. Since $\bigvee\limits_{i=1}^l x_i^{v_i} \preceq x$, then $ A\left(\bigvee\limits_{i=1}^l x_i^{v_i}\right) \subseteq A(x)$. Thus $A(x)= A\left(\bigvee\limits_{i=1}^l x_i^{v_i}\right)$.
Therefore $\sum\limits_{w\in P(1)}v_w(x) = \sum\limits_{w\in P(1)}v_w\left(\bigvee\limits_{i=1}^l x_i^{v_i}\right)$.
By the property \eqref{four} of a power lattice, this implies that $\rho(x) = \rho\left(\bigvee\limits_{i=1}^l x_i^{v_i}\right)$. Since $\bigvee\limits_{i=1}^l x_i^{v_i} \preceq x$, we must have $\bigvee\limits_{i=1}^l x_i^{v_i} = x$.

For the second part, if $F(x)= F(y)=\prod_{i=1}^l x_i^{v_{x_i}(x)}$, then the first part of the lemma says that $x=y=\bigvee\limits_{i=1}^l x_i^{v_{x_i}(x)}$.
\end{proof}


Now, we can show that the converse of Lemma \ref{lem:iniqualityvaluation} is also true.

\begin{corollary}
Let $(P,\preceq)$ be a power lattice. For $x,y\in P$, $x\preceq y$ if and only if $v_w(x)\leq v_w(y)$ for all $w\in P(1)$ (equivalently, $F(x) \subseteq F(y)$).
\end{corollary}
\begin{proof}
The first part of the proof is just Lemma \ref{lem:iniqualityvaluation}.
Conversely, suppose that for all $w\in P(1)$, $v_w(x)\leq v_w(y)$. Therefore, $w^{v_w(x)} \preceq w^{v_w(y)} \preceq \bigvee\limits_{w\in P(1)} w^{v_w(y)} = y$. The last equality is a consequence of Lemma \ref{lemma1}. This implies $\bigvee\limits_w w^{v_w(x)} \preceq y$ and by Lemma \ref{lemma1}, $x\preceq y$.
\end{proof}


\begin{definition}[Total order on the elements of same rank]\label{Def:totalorder}
Let $(P, \preceq, \rho )$ be a power lattice. For a positive integer $l$ with $P(l)\neq \emptyset$, we define a total order on $P(l)$ as follows. Let $x,y$ be two distinct elements of $P(l)$ with the corresponding factorizations $F(x)=(x_1,\dots,x_t)$ and $F(y)=(y_1,\dots,y_t)$. 

We say that $x\preceq_l y$ if
\[
\begin{cases}
x=y  \quad \text{or} \\
x \neq y \text{ and if } i \text{ is the smallest integer such that } x_i\neq y_i, \text{ then } x_i\prec_1 y_i.
\end{cases}
\]
\end{definition}
From Remark \ref{rem:1}, $F(x)$ and $F(y)$ in the previous definition have the same length. Moreover, Corollary \ref{lemma1} says that two distinct elements of $P(l)$ must have distinct factorization. The relation $\preceq_l$ on the factorizations is just the lexicographic ordering induced by $\preceq_1$, so it is clear that we have a total order.

We give an equivalent definition of the total order of Definition \ref{Def:totalorder}, that will be useful in the following sections.
\begin{lemma}\label{lem:0}
Let $P$ be a power lattice and let $x\neq y$ be two elements of $P(l)$. Then $x\prec_l y$ if and only if 
\[
\min_{\preceq_1} F(x)\backslash F(y) \prec_1 \min_{\preceq_1} F(y)\backslash F(x),
\]
where the multiset difference is done by also considering the multiplicity. 
\end{lemma}
\begin{proof}
Since $F(x)$ and $F(y)$ have the same length and they are distinct, then $F(x)\backslash F(y)$ and $F(y)\backslash F(x)$ are non-empty. Let $F(x)=(x_1,x_2,\dots,x_t)$ and $F(y)=(y_1,y_2\dots,y_t)$ be the factorizations of $x$ and $y$ respectively. Suppose that $j$ is the smallest integer such that $x_j\neq y_j$ and $x_j\prec_1 y_j$. Then $x_j\in F(x)\backslash F(y)$ and therefore $\min_{\preceq_1} F(x)\backslash F(y)\preceq_1 x_j$. Now $F(y)\backslash F(x)\subset \{y_j,\dots, y_t\}$. Therefore $y_j\preceq_1 \min_{\preceq_1} F(y)\backslash F(x)$. Since $x_j\prec_1 y_j$, we get the desired result.

Conversely, let $\min_{\preceq_1} F(x)\backslash F(y) \prec_1 \min_{\preceq_1} F(y)\backslash F(x)$. Suppose, $x \succ_l y $ and then it follows that $x_j \succ_1 y_j$. This implies $y_j \in F(y)\backslash F(x)$ and in fact, $y_j = \min_{\preceq_1} F(y)\backslash F(x)$. Combining all these we get,
\begin{equation}
\min_{\preceq_1} F(x)\backslash F(y) \prec_1 \min_{\preceq_1} F(y)\backslash F(x) = y_j \prec_1 x_j.
\end{equation}
Since all $x_i$ for $i < j$, precedes $y$, $\min_{\preceq_1} F(x)\backslash F(y) \prec_1 x_j$ gives a contradiction. This proves $x \prec_l y$.
\end{proof}

\begin{lemma}
Let $P$ be a power lattice. Let $x, y \in P$ be such that $x\in P(1)$ and $A(y)= \{x\}$. Then there is a unique chain, i.e., totally ordered set $x\covered x^2 \covered \cdots \covered y=x^{\rho(y)}$. The elements in this chain are all powers of $x$.
\end{lemma}
\begin{proof}

For the existence of such a chain, we first note that any element $z \in P$ such that $x \prec z \prec y = x^{\rho(y)}$ has to be in $\power{x}$. Indeed, since Lemma \ref{lem:iniqualityvaluation} implies $\{x\} = A(x) \subseteq A(z) \subseteq A(y) =\{x\}$ and thus $A(z) =\{x\}$. The existence of power of $x$ for each $r$ with $1 \le r \le \rho(y)$ follows from Definition \ref{def:rankedlattice} of ranked lattice.
The uniqueness follows from $(i)$ of Definition \ref{def:powerlattice} of a power lattice.
\end{proof}


\section{Complexes in power lattices}\label{sec:complexesinpowerlattices}
In this section, we introduce the notion of complexes in power lattices, that we call $\PP$-complexes, to generalize and extend the notion of abstract simplicial complexes, $q$-complexes, and multicomplexes. Then we define the notion of $\PP$-shellability of $\PP$-complexes extending the existing notions of shellability of the complexes mentioned above. The main result of this section is the shellability of the order complex of a $\PP$-shellable $\PP$-complex.


\begin{definition}[$\PP$-complexes]
Let $(P,\preceq)$ be a power lattice. A $\PP$-complex $S$ in $P$ is a non-empty subposet of $P$ such that $\forall x\in S$ and $y\in P$, $y\preceq x$ implies $y\in S$.
\end{definition}

In other words, a $\PP$-complex is a subset of $P$ which is closed under $\preceq$. 
For $Q\subseteq P$, the $\PP$-complex generated by $Q$ is the $\PP$-complex defined by
\[
\langle Q\rangle := \{x\in P\colon x\preceq y \text{ for some } y\in Q\}.
\]

\begin{definition}
Let $(P,\rho) $ be a power lattice and let $S$ be a $\PP$-complex in $P$. The elements of $S$ are called the faces of $S$ and the maximum elements of $S$ are called the facets of $S$. If all the facets have same rank, then $S$ is called a \emph{pure} $\PP$-complex. The \emph{rank} of a $\PP$-complex is defined to be the maximum of ranks of its faces.
\end{definition}

\begin{definition}[$\PP$-shellable $\PP$-complexes]\label{def:shellable}
Let $(P,\preceq, \rho)$ be a power lattice. A $\PP$-complex $S\subseteq P$ of rank $r$ is called $\PP$-shellable if $S$ is pure and there exists an ordering $f_1,f_2,\ldots,f_t$ on the facets of $S$ such that for all $1\leq i<j\leq t$, there exists $k<j$ for which $f_i \wedge f_j\preceq f_k\wedge f_j$ and $\rho(f_k\wedge f_j)=r-1$.
\end{definition}

If we take a finite set $A$ as in Example \ref{exa:1}, then we recover the notion of shellable simplicial complex.

\begin{definition}
    A simplicial complex $S$ is a non-empty set of subsets of $2^A$ such that  $S$ is closed under inclusion. A shellable simplicial complex is a simplicial complex $S$ such that its facets have the same cardinality $r$ and there exists an ordering $F_1,F_2,\cdots,F_t$ of all the facets of $S$ such that for all $i<j\leq t$, there exists $k<j$ such that $F_i \cap F_j\subseteq F_k\cap F_j$ and $|F_k\cap F_j|=r-1$.
\end{definition}

If we consider the power lattice $\Sigma(\Fq^n)$, then we have the $q$-shellable $q$-complexes

\begin{definition}[\cite{Alder10, GPR22}]
A $q$-complex $S$ is a non-empty set of subsets of $\Sigma(\Fq^n)$ such that  $S$ is closed under inclusion. A $q$-shellable $q$-complex is a $q$-complex $S$ such that its facets have the same dimension $r$ and there exists an ordering $F_1,F_2,\cdots,F_t$ of all the facets of $S$ such that for all $i<j\leq t$, there exists $k<j$ such that $F_i \cap F_j\subseteq F_k\cap F_j$ and $\dim_{\Fq} F_k\cap F_j=r-1$.
\end{definition}

\begin{definition}
Let $P$ be a power lattice of rank $n$. For $l\leq n-1$, an $l$-sphere $S_x$ of $P$ is a subposet such that there  exists $x\in P(l+1)$ such that $S_x=\{ y\in P\colon y\prec x \}$.
\end{definition}

\begin{remark}
When $P=2^{A}$ or $P=\Sigma(\Fq^n)$,  the $l$-spheres in $P$ are $\PP$-shellable.
\end{remark}

For the lattice $2^A$ of subsets, the shellability helps to describe the homology degrees of some topological space associated to the simplicial complex \cite[Appendix]{Bjo80}. For the lattice of subspaces $\Sigma(\Fq^n)$, the shellability helped partially to describe the homology degrees \cite{GPR22}. However, by considering the associated order complex, the homology can be fully obtained \cite{GPTVW24}. We generalize this results in the setting of $\PP$-complexes in a power lattice.

\begin{definition}[Order complex]
Let $P$ be a poset. A chain in $P$ is a totally ordered subset of $P$ and we write a chain as a sequence $(x_1\prec x_2\prec \dots \prec x_l)$. The order complex of $\mathcal{K}(P)$ is the poset of all chains in $P$.
\end{definition}


The order complex of a poset is known to be a simplicial complex. The next theorem says that the order complex of an $l$-sphere is shellable as a simplicial complex. For that we first need an ordering on the maximal chains.

\begin{definition}[Reverse lexicographic ordering]
Let $(P,\preceq)$ be a power lattice and let $S$ be a pure $\PP$-complex of rank $r$. Suppose that $\preceq_i$ is the total ordering on $P(i)$. The reverse ordering $\trianglelefteq$ on the facets of $\mathcal{K}(S)$ is defined as follows: Let $X=(x_0\preceq \dots \preceq x_r)$ and $Y=(y_0\preceq \dots \preceq y_r)$ be two maximal chains of $S$. $X \trianglelefteq Y$ if $X=Y$ or if $X\neq Y$ and $x_i\prec_i y_i$ where $i$ is the largest integer such that $x_i\neq y_i$.
\end{definition}
 It is clear that $\trianglelefteq$ is a total order, since this is just the reverse lexicographic order.

\begin{lemma}\label{lem:1}
Let $(P, \preceq, \rho)$ be a power lattice and let $(x_{i_1}\covered \dots \covered x_{i_s})$ be a chain in $P$ with $\rho(x_{i_j})=i_j$. If for all $i_1<i<i_s$, $x_i = \min_{\preceq_i}\{ a\in P(i)\colon x_{i-1} \prec a\prec x_{i+1} \}$, then for all $i_1<i\leq i_s$,
\[
\min_{\preceq_1} F(x_i)\backslash F(x_{i-1}) = \min_{\preceq_1} F(x_{i_s})\backslash F(x_{i-1}).
\]
\end{lemma}

\begin{proof}
Let $i_1<i< i_s$. First we show that \[
\min_{\preceq_1} F(x_{i+1})\backslash F(x_{i-1}) = \min_{\preceq_1} F(x_{i})\backslash F(x_{i-1}).
\]
Let $z = \min_{\preceq_1}  F(x_{i+1})\backslash F(x_{i-1}) $. Let $t$ be the smallest positive integer such that $z^t \npreceq x_{i-1}$ and therefore surely  $z^t \preceq x_{i+1}$. So $x_{i-1} \prec x_{i-1}\vee z^t \preceq x_{i+1}$ and by semimodularity of the lattice, we have $\rho(x_{i-1}\vee z^t) \leq \rho(x_{i-1}) + \rho(z^t) - \rho(x_{i-1}\wedge z^t) = (i-1) +t -(t -1)= i$. Since $\rho(x_{i+1}) = i+1$, we have $x_{i-1} \covered x_{i-1} \vee z^t \covered x_{i+1}$ If $x_{i-1} \vee z^t \neq x_i $, then $x_i \prec_i x_{i-1} \vee z^t$. Thus by Lemma \ref{lem:0},
\begin{equation}\label{min}
\min_{\preceq_1} F(x_i) \backslash F(x_{i-1}\vee z^t)  \prec_1 \min_{\preceq_1} F(x_{i-1}\vee z^t) \backslash F(x_i) \preceq_1 z. 
\end{equation}
We get the second inequality because $z^t \npreceq x_i$ as otherwise $x_{i-1} \vee z^t \preceq x_i$. And since they have the same rank, they would be equal which contradicts our assumption. Since $F(x_i) \backslash F(x_{i-1} \vee z^t) \subseteq F(x_i) \backslash F(x_{i-1})  \subseteq F(x_{i+1})\backslash F(x_{i-1})$, $$z = \min_{\preceq_1}F(x_{i+1})\backslash F(x_{i-1}) \preceq_1 \min_{\preceq_1} F(x_i) \backslash F(x_{i-1} \vee z^t).$$

But this together with Equation \eqref{min} implies $z \prec_1 z$, which is a contradiction. 

Therefore, $x_{i-1} \vee z^t = x_i$ and this implies $z^t \preceq x_i$. Thus $z\in F(x_i)\backslash F(x_{i-1})$. This implies that $\min_{\preceq_1}F(x_i)\backslash F(x_{i-1})\preceq_1 z$. On the other side, we know that $z=\min_{\preceq_1}  F(x_{i+1})\backslash F(x_{i-1}) \preceq_1 \min_{\preceq_1}  F(x_i)\backslash F(x_{i-1})$.
Hence $z = \min_{\preceq_1}F(x_i)\backslash F(x_{i-1})$. Thus it proves that
\[
\min_{\preceq_1} F(x_{i+1})\backslash F(x_{i-1}) = \min_{\preceq_1} F(x_{i})\backslash F(x_{i-1}),\quad \forall i, i_1<i< i_s.
\]

Let $i_1<i\leq i_s$. Now we show inductively that, for all $i<j\leq i_s$
\[
\min_{\preceq_1} F(x_{i+1})\backslash F(x_{i-1}) = \min_{\preceq_1} F(x_{j})\backslash F(x_{i-1}).
\]
The case $j=i_s$ will give the result of the Lemma.
The statement is obviously true for $j=i+1$. Suppose that it is true for $j>i$ i.e.  
\begin{equation}\label{ind:1}
\min_{\preceq_1} F(x_{i+1})\backslash F(x_{i-1}) = \min_{\preceq_1} F(x_{j})\backslash F(x_{i-1}).
\end{equation}
We want to show that 
\begin{equation}\label{ind:2}
\min_{\preceq_1} F(x_{i+1})\backslash F(x_{i-1}) = \min_{\preceq_1} F(x_{j+1})\backslash F(x_{i-1}).
\end{equation}
Since $x_j\prec x_{j+1}$, then, using Equation \eqref{ind:1}
\begin{equation}\label{ind:3}
z:=\min_{\preceq_1} F(x_{j+1})\backslash F(x_{i-1}) \preceq_1 \min_{\preceq_1} F(x_j)\backslash F(x_{i-1})=\min_{\preceq_1} F(x_{i+1})\backslash F(x_{i-1}).
\end{equation}
We claim that $z\in F(x_{i+1})\backslash F(x_{i-1})$. This implies that $\min_{\preceq_1} F(x_{i+1})\backslash F(x_{i-1})\preceq_1 z$ and therefore $\min_{\preceq_1} F(x_{j+1})\backslash F(x_{i-1})= \min_{\preceq_1} F(x_{i+1})\backslash F(x_{i-1})$ which concludes the proof.

\textbf{Proof of claim}: Suppose that $z\notin F(x_{i+1})\backslash F(x_{i-1})$. Thus $v_z(x_{i-1})=v_z(x_i)=v_z(x_{i+1})=r$. Since $v_z(x_{j+1})>v_z(x_{i-1})$ and $i+1<j+1$, then we can find $j> k\geq i$ such that $v_z(x_{k-1})=v_z(x_{k})=v_z(x_{k+1})= r$ and $v_z(x_{k+2})>r$.

Now, take $x_k\preceq x_k\vee z^{r+1} \preceq x_{k+2}$. Since $z^{r+1}\npreceq x_k$, then $x_k\prec x_k\vee z^{r+1}$. Furthermore, by semimodularity, $\rho(x_k\vee z^{r+1})\leq \rho(x_k)+\rho(z^{r+1})-\rho(x_k\wedge z^{r+1})=k+1<\rho(x_{k+2})$. Hence $x_k\vee z^{r+1} \prec x_{k+2}$. Hence $x_k\prec x_k\vee z^{r+1} \prec x_{k+2}$. By the minimality of $x_{k+1}$ in this interval, $x_{k+1}\preceq_{k+1} x_k\vee z^{r+1}$. Since $z^{r+1}\npreceq x_{k+1}$, then $x_{k+1}\prec_{k+1} x_k\vee z^{r+1}$. By Lemma \ref{lem:0}, 
\begin{equation}\label{ind:4}
\min_{\preceq_1} F(x_{k+1})\backslash F(x_k\vee z^{r+1}) \prec_1 \min_{\preceq_1} F(x_k\vee z^{r+1})\backslash F(x_{k+1}).
\end{equation}
Since $z^{r+1}\npreceq x_{k+1}$, then $z\in F(x_k\vee z^{r+1})\backslash F(x_{k+1})$. Thus $\min_{\preceq_1} F(x_k\vee z^{r+1})\backslash F(x_{k+1})\preceq_1 z$. Combining this with Equation \eqref{ind:4}, we have
\begin{equation}\label{ind:5}
\min_{\preceq_1} F(x_{k+1})\backslash F(x_k\vee z^{r+1}) \prec_1 z.
\end{equation}

Since $k<j$, $x_{k+1}\preceq x_j$. Therefore $F(x_{k+1})\subseteq F(x_j)$. This implies that $F(x_{k+1})\backslash F(x_k\vee z^{r+1})\subseteq F(x_j)\backslash F(x_k\vee z^{r+1})$. Thus
\[
\min_{\preceq_1} F(x_j)\backslash F(x_k\vee z^{r+1}) \preceq_1 \min_{\preceq_1} F(x_{k+1})\backslash F(x_k\vee z^{r+1}).
\]
Together with Equation \eqref{ind:5}, this gives us
\begin{equation}\label{ind:6}
\min_{\preceq_1} F(x_j)\backslash F(x_k\vee z^{r+1})  \prec_1 z.
\end{equation}
Now, since $i\leq k$, then  $x_{i-1}\preceq x_{i}\preceq x_k\preceq x_k\vee z^{r+1}$. Thus $F(x_j)\backslash F(x_k\vee z^{r+1})\subseteq F(x_j)\backslash F(x_{i-1})$. Thus $\min_{\preceq_1} F(x_j)\backslash F(x_{i-1}) \preceq_1 \min_{\preceq_1} F(x_j)\backslash F(x_k\vee z^{r+1})$. Together with Equation \eqref{ind:3}, this gives us  $z \preceq_1 \min_{\preceq_1} F(x_j)\backslash F(x_k\vee z^{r+1})$. Combining this with Equation \eqref{ind:6}, we have $z\prec_1 z$ which is impossible. Hence the claim.
\end{proof}

If $P$ is the lattice of subsets or lattice of subspaces, the rank function $\rho$ is modular. In this case, every sphere is shellable as $P$ is a modular lattice. For general power lattices, we only have semimodularity of the rank function $\rho$. So we cannot directly conclude whether the spheres are shellable $\PP$-complexes or not.  
{ However, we can directly show the shellability of the associated order complexes.}

\begin{theorem}[Shellability of the order complex of a sphere]\label{shellabilitysphere}
Let $P$ be a power lattice of rank $n$ and let $l$ be an integer such that $l\leq n-1$. For an $l$-sphere $S_x$ in $P$, the order complex $\mathcal{K}(S_x)$ is a shellable simplicial complex with respect to the ordering $\trianglelefteq$ on the facets of $\mathcal{K}(S_x)$.
\end{theorem}
\begin{proof}
Suppose that $S_{x}=S_{z_{l+1}}$ where $z_{l+1}\in P(l+1)$. Let $x_l,y_l\in P(l)$ such that $x_l\prec z_{l+1}$ and $y_l\prec z_{l+1}$.
Suppose that $X=(x_0\prec \dots \prec x_l)$ and $Y=(y_0\prec \dots \prec y_l)$ are two distinct maximal chains in $P$ such that $X\triangleleft Y$. We want to find a chain $U\triangleleft Y$ such that $X\cap Y\subset U\cap Y$ and $|U\cap Y|=l$.

Let $x_{l+1}=y_{l+1}=z_{l+1}$. Let $s$ be the largest integer such that $s\leq l$ and $x_s\neq y_s$.
Therefore, $x_s\prec_s y_s$ and $x_{s+1}=y_{s+1}$. 
Since $x_0=y_0$, we can find integer $r$ such that $0\leq r<s$, $x_r=y_r$ and $x_i\neq y_i$ for $r< i \leq s$. If $r=s-1$, then take the chain 
\[
U=(y_0\prec \dots \prec y_{s-1} \prec x_{s} \prec y_{s+1}\prec \dots \prec y_l).
\]
Otherwise, assume that $r\leq s-2$. We claim that $y_i\neq \min_{\preceq_{i}}\{ a\in P\colon y_{i-1}\prec a\prec y_{i+1} \}$ for some $r<i\leq s$. In this case we take the chain
\[
U=(y_0\prec \dots \prec y_{i-1}\prec  \min_{\preceq_{i}}\{ a\in P\colon y_{i-1}\prec  a\prec y_{i+1} \} \prec y_{i+1}\prec \dots \prec y_l).
\]
That would conclude the proof.

\textbf{Proof of claim}: Assume that for all $i$ such that $r<i\leq s$, $y_i= \min_{\preceq_{i}}\{ a\in P\colon y_{i-1}\prec a\prec y_{i+1} \}$.

Let $j$ be the largest integer such that $r\leq j < s$ and $y_j\preceq x_{s}$. Such $j$ exists since $y_r=x_r\prec x_s$.

First, since $y_j\prec y_s\prec y_{s+1}$, then $F(y_j)\subset  F(y_s)\subset  F(y_{s+1})$. Therefore
\begin{equation}\label{1}
    \min_{\preceq_{1}}F(y_{s+1}) \backslash F(y_j) \preceq_1  \min_{\preceq_{1}} F(y_{s+1}) \backslash F(y_s)
\end{equation}
Now, since $x_s\prec x_{s+1} = y_{s+1}$, then $F(x_s)\subset F(y_{s+1})$. Therefore,
\begin{equation}\label{2}
    \min_{\preceq_1} F(y_{s+1})\backslash F(y_s) \preceq_{1} \min_{\preceq_1}F(x_s)\backslash F(y_s)
\end{equation}

Equations \eqref{1} and \eqref{2} imply that 

\begin{equation}\label{3}
    \min_{\preceq_{1}}F(y_{s+1}) \backslash F(y_j) \preceq_{1} \min_{\preceq_1}F(x_s)\backslash F(y_s)
\end{equation}

Now, since $x_s\prec_s y_s$, by Lemma \ref{lem:0}, $\min_{\preceq_1}F(x_s)\backslash F(y_s)\prec_1 \min_{\preceq_1}F(y_s)\backslash F(x_s)$ and together with Equation \eqref{3}, this implies that
\begin{equation}\label{4}
    \min_{\preceq_{1}}F(y_{s+1}) \backslash F(y_j) \prec_{1} \min_{\preceq_1}F(y_s)\backslash F(x_s)
\end{equation}
Now, by Lemma \ref{lem:1}, we have $a:=\min_{\preceq_{1}}F(y_{s+1}) \backslash F(y_j) = \min_{\preceq_{1}}F(y_{j+1}) \backslash F(y_j)$ and the second equality implies that $v_a(y_{j}) < v_a(y_{j+1}) \leq v_a(y_s)$.
Hence $a\in F(y_{j+1})\subseteq F(y_s)$. If $a\in F(y_s)\backslash F(x_s)$, then $\min_{\preceq_1} F(y_s)\backslash F(x_s) \preceq_1 a$ i.e.
\[
\min_{\preceq_1} F(y_s)\backslash F(x_s) \preceq_1 \min_{\preceq_{1}}F(y_{s+1}) \backslash F(y_j).
\]
Together with Equation \eqref{4}, this gives $\min_{\preceq_1} F(y_s)\backslash F(x_s) \prec_{1} \min_{\preceq_1}F(y_s)\backslash F(x_s)$. 
This is impossible and therefore, $a\notin F(y_s)\backslash F(x_s)$. 
This implies that $v_a(y_s) \leq v_a(x_s)$ and thus, $v_a(y_j) <v_a(y_{j+1})\leq v_a(y_s) \leq v_a(x_s)$. Since $a\preceq y_s$, this also implies that $a\preceq x_s$. Let $t = v_a(y_{j+1})$ and therefore, $a^t \npreceq y_j$, but $a^t \preceq y_{j+1} $ and $a^t \preceq x_s$.
Since $y_j\prec x_s$, then $a^t\vee y_j \preceq x_s$. 
Now $y_j\preceq a^t\vee y_j\preceq y_{j+1}$. The first inequality is strict because $a^t\npreceq y_j$ and thus comparing the ranks, we must have  $a^t\vee y_j= y_{j+1}$. 
Therefore, $y_{j+1}\preceq x_s$ and 
by the property of $j$, we have $j+1\geq s$. Hence $y_s\preceq y_{j+1}\preceq x_s$.
Since $\rho(y_s)=\rho(x_s)$, then $x_s=y_s$ which is a contradiction. 
Hence we conclude that there exists $r<i\leq s$ such that $y_i\neq \min_{\prec_{i}}\{ a\in P\colon y_{i-1}\preceq a\prec y_{i+1} \}$ and we are done.
\end{proof}

\begin{definition}
Let $P$ be a power lattice and suppose $S$ is a shellable $\PP$-complex of rank $r$ with the linear ordering of facets given by $\preceq_r$. We define a linear ordering $\ll$ on the maximal chains of $S$ as follows. Suppose $X=(x_0\prec x_1\prec \dots \prec x_r)$ and $Y=(y_0\prec y_1\prec \dots \prec y_r)$ with $X\neq Y$. Define $X\ll Y$ if
either $x_r\prec_l y_r$ or  
    \[ x_r=y_r, \quad  \text{ and }  \quad (x_0\preceq \dots \preceq x_{r-1})\triangleleft (y_0\preceq \dots \preceq y_{r-1}).
\]
\end{definition}

This is again a total order. Now, we show the main theorem of this section relating the shellability of a $\PP$-complex to the shellability of the associated order complex.

\begin{theorem}\label{thm:shellabilityordercompex}
Let $P$ be a power lattice. If $S$ is a shellable $\PP$-complex with the linear ordering of facets given by $\preceq_r$, then the order complex $\mathcal{K}(S)$ is a shellable simplicial complex whose facets are ordered with $\ll$.
\end{theorem}
\begin{proof}
Let $X=(x_0\prec \dots \prec x_r)$ and $Y=(y_0\prec \dots \prec y_r)$ be two distinct maximal chains in $P$ such that $X\ll Y$. Let $e$ be the largest index such that $x_e\neq y_e$.

\textbf{Case 1.} Suppose that $e<r$ and thus $x_{e+1} = x_{e+1}$, $e\leq r$. Let $f$ be the largest index smaller than $e$ such that $U_f = V_f$. Such $f$ exists because $U_0 = V_0$. Consider the two chains
\begin{align*}
    &\overline{X}=(y_0\preceq \dots \preceq y_{f}\preceq x_{f+1} \preceq \dots \preceq x_{e})\\
    &\overline{Y}=(y_0\preceq \dots \preceq y_{f}\preceq y_{f+1} \preceq \dots \preceq y_{e}) 
\end{align*}
We see that $\rho(x_{e})=\rho(y_{e})=e$ and $x_e\preceq x_{e+1}$, $y_e\preceq x_{e+1}$, with $\rho(x_{e+1})=e+1$. 
Applying Theorem \ref{shellabilitysphere}, we can find a chain $\overline{U}=(u_0\preceq \dots \preceq u_{e})$ such that $\overline{X}\cap\overline{Y}\subset \overline{U}\cap \overline{Y}$, $|\overline{U}\cap \overline{Y}|=e$ and $\overline{U}\triangleleft \overline{Y}$. Take $U=(u_0\preceq \dots \preceq u_{e}\preceq y_{e+1}\preceq \dots \preceq y_{r})$ and we have $X\cap Y\subset U\cap Y$, $|U\cap Y|=r$ and $U\ll Y$.

\textbf{Case 2.} Suppose that $e=r$, we have $x_r\prec_r y_r$. Then let $f$ be the largest index such that $x_f=y_f$. Again $f$ exists because $x_0 = y_0$. Now, by the shellability of $S$, there is $z_r$ such that $z_r\prec_r y_r$, $\rho(z_r\wedge y_r)=r-1$ and $x_r\wedge y_r\subset z_r\wedge r_r$. Let $t$ be the largest index such that $y_t\prec z_r$. Then $t\geq f$. Moreover, since $z_r\neq y_r$ and $\rho(z_r)=\rho(y_r)$, then $y_r\npreceq z_r$ i.e. $t<r$.
Now, take the chains
\begin{align*}
    &Z=(y_0\prec \dots \prec y_{f}\prec \dots \prec y_{t}\prec z_{t+1}\prec \dots \prec z_{r}),\\
    &Y=(y_0\prec \dots \prec y_{f}\prec \dots \prec y_{t}\prec y_{t+1}\prec \dots \prec y_{r}).
\end{align*}
Since $y_r\neq z_r$, then $y_r\prec y_r\vee z_r$. Therefore $r=\rho(y_r)<\rho(y_r\vee z_r)\leq \rho(y_r)+\rho(z_r)-\rho(y_r\wedge z_r)$ (because of the semimodularity of $P$). Since $\rho(z_r) = \rho(y_r) = r$ and $\rho(z_r\wedge y_r) = r-1$, then $r<\rho(y_r\vee z_r)\leq r+1$. Therefore $z_r$ and $y_r$ precede $y_r+z_r$ which is of rank $r+1$. Thus we can apply Theorem \ref{shellabilitysphere} to find a chain $U=(u_0\preceq \dots \preceq u_{e})$ such that $Z\cap  Y\subset U\cap Y$, $|U\cap Y|=e$ and $U\ll Y$. This concludes the proof.
\end{proof}

The main consequence of this theorem is that the homotopy type of a shellable $\PP$-complex is understood. 

As topological spaces, a finite poset and its associated order complex are weakly homotopy equivalent \cite{McCord}. Moreover,  a shellable simplicial complex has the same homotopy type as a wedge of spheres, i.e., the reduced simplicial homology is zero except at maximal dimension \cite[Appendix]{Bjo80}. Therefore, by Theorem \ref{thm:shellabilityordercompex}, we have the following corollary that generalizes \cite{GPTVW24} and \cite[Corollary 3.10]{Cim06}.
\begin{corollary}
Let $P$ be a power lattice. If $S$ is a shellable $\PP$-complex of rank $r$ with the linear ordering of facets given by $\preceq_r$, then it has the homotopy type of a wedge of spheres.
\end{corollary}

As mentioned in the introduction, another way to obtain shellable simplicial complexes is to start with EL-shellable or CL-shellable posets. The order complexes associated to these posets give us a shellable simplicial complex \cite{Bjo80,BW83}. Our condition is different as instead of using the EL/CL-shellability, we are using the $\PP$-shellability of the poset as a $\PP$-complex in a power lattice. Shellable simplicial complexes are of interests as they have nice topological and algebraic properties and the construction via the order complex provides new classes of shellable simplicial complexes.
\section{Matroids in power lattices}\label{sec:matroidsinpowerlattices}

In the classical theory of matroids, the independent sets of a matroids form a shellable simplicial complex. To continue with the analogy, we want to provide constructions of shellable $\PP$-complexes via a generalization of matroids. Then we show how to get shellable $\PP$-complexes from the matroids. 

\begin{definition}
Let $(P,\preceq)$ be a power lattice with rank function $\rho$. A matroid $(P,\I)$  on $P$ is defined by a subset $\I$ of $P$ such that
\begin{enumerate}[({I}1)]
    \item\label{I1} $\0\in \I$,
    \item\label{I2} if $x\preceq y$ and $y\in \I$, then $x\in \I$,
    \item\label{I3} For all $x,y\in \I$ such that $\rho(x)<\rho(y)$, there exists an atom $a\in P$ such that $v_a(x) < v_a(y)$, $x\vee a^{v_a(x)+1} \in \I$.
\end{enumerate}
\end{definition}

\begin{remark}
    Notice that, this definition is more general to the definition of $q$-matroids in the lattice of subspaces \cite{JP18}. However, the axioms in the definition were enough to get shellability of both for matroid and $q$-matroid \cite{GPR22,Bjo92}.
\end{remark}

\begin{example}\
\begin{enumerate}
    \item Take $(P,P)$. This is a trivial matroid.
    \item  More generally we can define the uniform matroids $U_k(P)$. Let $P$ be a power lattice with $\rho(\1)=n$. Suppose $k\leq n$ and let $\I=U_k(P):=\{x\in P\colon \rho(x)\leq k\}$. Then (I1) and (I2) are clear. To check (I3). Suppose that $\rho(x)<\rho(y)\leq k$ and hence $\rho(x)\leq k-1$. There exists an atom $a\in P$ such that $v_a(x)<v_a(y)$. Now take $x\vee a^{v_a(x)+1}$. 
    By semimodularity, we have
    \begin{align*}
    \rho(x\vee a^{v_a(x)+1})&\leq \rho(x)+\rho(a^{v_a(x)+1})-\rho(x\wedge a^{v_a(x)+1})\\
    &\leq \rho(x)+\rho(a^{v_a(x)+1})-\rho(a^{v_a(x)})\\
    &\leq \rho(x)+1\\
    &\leq k
    \end{align*}
    Hence $x\vee a^{v_a(x)+1}\in \I$.
\end{enumerate}
\end{example}

\begin{example}\
    \begin{enumerate}[(i)]
        \item The classical matroids on the lattice of subsets \cite{Oxley}. They consist of a set $I$ of a subsets of subsets of a finite set. $I$ is closed under inclusion and for two subsets $A,B\in I$ such that $|A|<|B|$, there exists $x\in B\setminus A$ such that $A\cup \{x\}\in I$.
        \item The relatively recent $q$-matroids on the lattice of subspaces \cite{JP18}. They consist of a set $I$ of subspaces of a vector space $\Fq^n$. $I$ is closed under inclusion and for two subspaces $A,B\in I$ such that $\dim A<\dim B$, there exists $x\in B\setminus A$ such that $A\cup \{x\}\in I$.
        \item The discrete polymatroids on the lattice of multiset subsets \cite{HH02}. They consist of a set $I$ of multiset subsets of a finite multiset. $I$ is closed under inclusion and for two subsets $A,B\in I$ such that $|A|<|B|$, there exists $x\in B\setminus A$ (the difference takes into account the multiplicity) such that $A\cup \{x\}\in I$, the last union means that the multiplicity of $x$ in $A$ is increased by $1$.
        \item The sum-matroids on the Cartesian product of lattice of subspaces. Although in \cite{PPR23}, the sum matroids were defined by rank functions, one can show that by taking the independent elements, i.e. the elements such that whose rank is the same as the total dimension, we get a matroid over the Cartesian product of lattice of subspaces.
    \end{enumerate}
\end{example}

\begin{definition}
    Let $(P,\preceq)$ be a power lattice with rank function $\rho$ and let $(P,\I)$ be a matroid in $P$, the elements of $\I$ are called the independent elements of $(P,\I)$. A maximum element in $\I$ with respect to $\preceq$ is called a basis of $(P,\I)$. Let $\B$ denote the set of all bases of $(P,\I)$.
\end{definition}

\begin{lemma}\label{lem:samerank}
    Let $(P,\preceq)$ be a power lattice with rank function $\rho$ and let $(P,\I)$ be a matroid in $P$. Let $\B$ be the set of bases of the matroid. All elements of $\B$ have the same rank.
\end{lemma}
\begin{proof}
    Let $x,y\in \B$. Suppose $\rho(x)<\rho(y)$. By (I\ref{I3}), there exists an atom $a\in P$ such that $v_a(x) < v_a(y)$, $x\vee a^{v_a(x)+1} \in \I$. It is clear that $x\preceq x\vee a^{v_a(x)+1}$ and by Lemma \ref{lem:valuation}, $v_a(x)<v_a(x\vee a^{v_a(x)+1})$ and this implies $x\neq x\vee a^{v_a(x)+1}$. This contradicts the maximality of $x$ in $\I$. Therefore $\rho(x)=\rho(y)$.
\end{proof}

Similar to classical matroids and $q$-matroids, the bases of a matroid satisfies the following nice properties.
\begin{theorem}
    Let $(P,\preceq)$ be a power lattice with rank function $\rho$ and let $(P,\I)$ be a matroid in $P$. Let $\B$ be the set of bases of the matroid. Then
    \begin{enumerate}[({B}1)]
    \item\label{B1} $\B\neq \emptyset$,
    \item\label{B2} For all $x,y\in \B$, if $x\preceq y$, then $x=y$,
    \item\label{B3} For all $x,y\in \B$, and for each $u\preceq x$ with $\rho(u)=\rho(x)-1$ and $x\wedge y\preceq u$, there exists an atom $a$ such that $v_a(u) < v_a(y)$ and $u\vee a^{v_a(u)+1}\in \B$.
\end{enumerate}
\end{theorem}
\begin{proof}\
    \begin{enumerate}[({B}1)]
    \item $\I$ is non-empty and it is finite so there must be a least one maximal element.
    \item Clear.
    \item Let $x,y\in \B$, and let $u\preceq x$ with $\rho(u)=\rho(x)-1$ and $x\wedge y\preceq u$.
    By (I\ref{I2}) $u\in \I$. Since $\rho(u)<\rho(x)$, by Lemma \ref{lem:samerank}, $\rho(u)<\rho(y)$. Therefore, by (I\ref{I3}),  there exists an atom $a\in P$ such that $v_a(u) < v_a(y)$, $u\vee a^{v_a(u)+1} \in \I$. Now $u\prec u\vee a^{v_a(u)+1}$ and $\rho(u)=\rho(x)-1$. Therefore $u\vee a^{v_a(u)+1}$ is a basis.
    \end{enumerate}
\end{proof}

\begin{remark}
    Again, the axioms are more general than in the case of $q$-matroids.
\end{remark}


Now we show the dual basis exchange property for matroids.

\begin{theorem}[Dual basis exchange property]\label{dualbasisexchange}
    Let $(P,\preceq)$ be a power lattice with rank function $\rho$ and let $(P,\I)$ be a matroid in $P$. Let $\B$ be the set of bases of the matroid. Let $x,y\in \B$ such that $x\neq y$ and let $a$ be an atom with $v_a(y)> v_a(x)$. Then there exist $u\in P$ with $\rho(u)=\rho(x)-1$ and an atom $b$ such that $v_b(y)< v_b(x)$, $x\wedge y\preceq u$, $x=u\vee b^{v_b(u)+1}$ and $u\vee a^{v_a(u)+1}$ is a basis.
\end{theorem}
\begin{proof}
    Let $r=\rho(x)$ and let $s=r-\rho(x\wedge y)$. If $s=1$, then $\rho(x\wedge y)=r-1$. Take $u=x\wedge y$. Then we observe that
    \begin{enumerate}[(i)]
        \item $x\wedge y\preceq u$ is true.
        \item  Since $x\neq y$ and $\rho(x)=\rho(y)$, there exists $b$ such that $v_b(y)<v_b(x)$. Moreover, we have $x\wedge y\preceq u\prec u\vee b^{v_b(u)+1}$. Since $v_b(u)\leq v_b(y)$, $v_b(u)<v_b(x)$. Therefore $ b^{v_b(u)+1}\preceq x$. Hence $x\wedge y\preceq u\prec u\vee b^{v_b(u)+1}\preceq x$. Since $\rho(x\wedge y)=r-1$ and $\rho(x)=r$, we have $u\vee b^{v_b(u)+1}= x$.
        \item Similarly, we can check that $u\vee a^{v_b(a)+1}= y$. 
    \end{enumerate}
    We proved that the statement holds for $s=1$. We now show the statement by induction on $s$. Assume that the statement holds for $r-\rho(x\wedge y)<s$, where $s>1$. We want to show that the statement holds for $r-\rho(x\wedge y)=s$. 
    
    Suppose that $r-\rho(x\wedge y)=s\geq 2$. Therefore $x\wedge y \prec w\prec y$ for some $w$ such that $\rho(w)=r-1$ and $v_a(w)>v_a(x\wedge y)$. By (B\ref{B3}), there is an atom $z$ such that $v_z(w) < v_z(y)$ and $t=w\vee z^{v_z(w)+1}$ is a basis. Now, by Lemma \ref{lem:valuation}, $v_z(x\wedge y)=\min\{v_z(x),v_z(y)\}$. Hence $\min\{v_z(x),v_z(y)\}\leq v_z(w)<v_z(y)$ and therefore $v_z(x\wedge y)=v_z(x)$. This gives us $v_z(x)\leq v_z(w)<v_z(t)$. So $x\neq t$ and $x\wedge y \prec x\wedge t$.
    Moreover, $v_a(y)\geq v_a(w)>v_a(x\wedge y)$. By Lemma \ref{lem:valuation}, $v_a(x\wedge y)=v_a(x)$. In addition $v_a(t)\geq v_a(w)>v_a(x\wedge y)=v_a(x)$ so that $v_a(t)>v_a(x)$.  By the induction hypothesis, there exist $u$ with $\rho(u)=\rho(x)-1$ and an atom $b$ such that $v_b(t)< v_b(x)$, $x\wedge t\preceq u$, $x=u\vee b^{v_b(u)+1}$ and $u\vee a^{v_a(u)+1}$ is a basis. We easily check that $x\wedge y\preceq x\wedge t\preceq u$. Finally, $x\wedge y\preceq x\wedge t$ implies that $v_b(x\wedge y)\leq v_b(x\wedge t)$. Since $v_b(t)< v_b(x)$, then $v_b(x\wedge y)\leq v_b(t)$. If $v_b(x)\leq v_b(y)$, then $v_b(x)=\min\{v_b(x),v_b(y)\}=v_b(x\wedge y)\leq v_b(t)$ which is a contradiction. Therefore $v_b(y)< v_b(x)$.
    
\end{proof}

\begin{theorem}
    Let $(P,\preceq)$ be a power lattice with rank function $\rho$ and let $(P,\I)$ be a matroid in $P$. Let $\B$ be the set of bases of the matroid whose elements have rank $r$. Then $\I$ is shellable where the elements of $\B$ are ordered with $\preceq_r$.
\end{theorem}
\begin{proof}
    Let $x$ and $y$ be two bases such that $x\prec_r y$. Let $F(x)=(x_1,\dots,x_t)$ and $F(y)=(y_1,\dots,y_t)$ be the factorization of $x$ and $y$ respectively. We also assume that these factorizations are already written in such a way that $x_i\preceq_1 x_{i+1}$ and $y_i\preceq_1 y_{i+1}$. Therefore
    \[
    \min_{\preceq_1} F(x)\setminus F(y)\prec_1 \min_{\preceq_1} F(x)\setminus F(y).
    \]
    Let $i_0$ be the smallest integer such that $x_{i_0}\neq y_{i_0}$. In this case $x_{i_0}\prec_1 y_{i_0}$ and $v_{x_{i_0}}(x)>v_{x_{i_0}}(y)$. By Theorem \ref{dualbasisexchange}, There exists $u$ with $\rho(u)=\rho(x)-1$ and $y_j$ such that $v_{y_j}(x)<v_{y_j}(y)$, $x\wedge y\preceq u$ and $y=u\vee y_j^{v_{y_j}(u)+1}$ and $z=u\vee x_{i_0}^{v_{x_{i_0}}(u)+1}$ is a basis. We have $x\wedge y\preceq u\preceq z\wedge y$. The next step is to show that $z\prec_r y$.

    For clarity, write
    \[
    F(x)=\prod_{i=1}^s a_i^{v_{a_i}(x)}, \quad a_i\prec a_{i+1}
    \]
    and 
    \[
    F(y)=\prod_{i=1}^s b_i^{v_{b_i}(y)}, \quad b_i\prec b_{i+1}
    \]
    and $i_1$ is the smallest integer such that $v_{a_i}>v_{b_i}$. Notice that $x_{i_0}=a_{i_1}$. For $z$, we have $(\bigvee_{i=1}^{i_1-1} b_i^{v_{b_i}(y)})\vee b_{i_1}^{v_{b_{i_1}}+1}\prec z$. Moreover $z$ may contain a new atom $a$ which is not equal to $b_i$, $i=1,\dots,i_1-1$. But in any case, whether $a\prec_1 b_{i_1}$ or not, we always have $z\prec_r y$. This also implies that $x\wedge y\prec z$ so that $\rho(x\wedge y)\leq \rho(u)\leq \rho(z\wedge y)<\rho(z)$. Hence $\rho(z\wedge y)=r-1$.
\end{proof}

\subsection{Weighted graphic matroids in the lattice of multiset subsets}

In this subsection, we consider the particular power lattice of multiset subsets of a finite multiset. We generalize the classical graphic to matroids on the lattice of multiset subsets from weighted graphs. 

Let $S$ be a finite multiset i.e., some elements can be repeated. For simplicity, we write $S=\prod_{i=1}^l x_i^{n_i}$ where the elements of $S$ are the $x_i$'s with the corresponding multiplicities denoted by the $n_i$'s. These elements of $S$ can be thought as variables whereas $S$ can be thought as a monomial. The lattice $P$ of subset of $S$ is the set of all multiset subsets of $S$ (which can be thought as monomials) i.e.
\[
P=\left\lbrace \prod_{i=1}^l x_i^{a_i}\colon 0\leq a_i\leq n_i\right\rbrace.
\]
The atoms of $P$ are the variables $x_i$. The maximal element is $\1=S$ and the minimal element is $\0=1$. The valuation of a monomial $\prod x_j^{a_j}$ at a variable $x_i$ is $v_{x_i}(\prod x_j^{a_j})=a_i$. The meet and join between two monomials are given by 
\[
(\prod x_i^{a_i}) \cap(\prod x_i^{b_i})=\prod x_i^{\min\{a_i,b_i\}}\text{ and } 
(\prod x_i^{a_i}) \cup(\prod x_i^{b_i})=\prod x_i^{\max\{a_i,b_i\}}.
\]
In fact the meet is just the greatest common divisors of two monomials whereas the join is the least common multiple.
The rank function $\rho$ is the same as the sum of the valuation at every atoms. It is not difficult to show that we have a power lattice with the partial ordering defined by multiset inclusion, which we denote by $\subseteq$. In the remaining part of this paper, whenever we talk about a lattice of multiset subsets, we always use the above notations.

When $P$ is the lattice of multiset subsets of a finite multiset, the $\PP$-complexes in $P$ are known as multicomplexes \cite{Stan}.

{\begin{definition}
 Let $P$ be the lattice of multiset subsets of $S = \prod_{i=1}^l x_i^{n_i}$. A multicomplex $\Delta$ in $S$ is a subset of $P$ such that
    \begin{enumerate}[(i)]
        \item $\Delta\neq \emptyset$,
        \item if $x\in \Delta$ and $y | x$, then $y\in \Delta$.
    \end{enumerate}
\end{definition}}

Here we recall the definition of shellable multicomplexes from \cite{Cim06} that is also obtained as the particular case of general shellable $\PP$-complexes as we defined in Definition \ref{def:shellable}. 

\begin{definition}\cite[Definition 3.4]{Cim06}
       A pure shellable multicomplex is a multicomplex $S$ such that its facets have the same cardinality $r$ when counted with multiplicity and there exists an ordering $F_1,F_2,\cdots,F_t$ of all the facets of $S$ such that for all $i<j\leq t$, there exists $k<j$ such that $F_i \cap F_j\subseteq F_k\cap F_j$ and $|F_k\cap F_j|=r-1$, where the cardinality is counted with multiplicity.
\end{definition}
As we have seen in a previous section, one way to get a shellable multicomplex is to build it from a matroid.
For clarity, let us recall the definition of matroids in the lattice of multiset subsets.

\begin{definition}
Let $P$ be a lattice of multiset subsets of $S=\prod_{i=1}^l x_i^{n_i}$ with rank function $\rho$. A matroid $(P,\I)$  on $P$ is defined by a subset $\I$ of $P$ such that
\begin{enumerate}[({I}1)]
    \item\label{I11} $\0\in \I$,
    \item\label{I22} if $x\subseteq y$ and $y\in \I$, then $x\in \I$,
    \item\label{I33} For all $x,y\in \I$ such that $\rho(x)<\rho(y)$, there exists $x_i$ such that $v_{x_i}(x) < v_{x_i}(y)$, $x\cup x_i^{v_{x_i}(x)+1} \in \I$.
\end{enumerate}
\end{definition}

Matroids on multisets are already known as discrete polymatroids. One can have a look at the paper \cite{HH02} and the properties (D1) and (D2) in the Introduction. Various constructions of matroids on multisets therefore already exist (see also \cite{BCF23}). In the following, we present a construction similar to graphical matroids, using weighted graphs. 

\begin{definition}[Weighted graphs]
A weighted graph $G=(V,E)$ is a set of vertices $V$ together with a subset $E\subseteq V\times V$ such that to each $e\in E$ is attached a positive integer $wt_G(e)$, called the weight of $e$.
\end{definition}

\begin{definition}
Let $G=(V,E)$ be a weighted graph. A subgraph of $G$ is a weighted graph $G_1=(V_1,E_1)$ such that $V_1\subseteq V$, $E_1\subseteq E$ and for $e\in E_1$, $0<wt_{G_1}(e)\leq wt_G(e)$. For a set of edges $E_1$ we denote by $V(E_1)$ the set of vertices which are either an origin or a tail of an edge in $E_1$.
\end{definition}

\begin{definition}[Cycles]
    Given a weighted graph $(V,E)$, a cycle of $(V,E)$ is a subgraph $G_1=(V(E_1),E_1)$ where $E_1=\{e_0,\dots,e_n\}\subset E$ such that $wt_{G_1}(e_i)=wt_G(e_i)$ and the tail of $e_{i -1 \mod n}$ is the origin of $e_{i \mod n}$ for any $i$.
\end{definition}

In other words, a cycle is similar to the classical notion of cycle on simple graph except that the weight of every edge in the cycle must be maximal, i.e. the same as its weight in the original graph.

\begin{definition}[Lattice of multiset subsets associated to a weighted graph]
    Let $G=(V,E)$ be a weighted graph. Let $S=\prod_{e\in E} x_e^{wt_G(e)}$ be a multiset on the edges in $E$.
    Define $P$ be the lattice of multiset subsets of $S$. For a multiset subset $H$ in $P$, the subgraph induced by $H$ is the weighted graph $(V(H),E(H))$, where $E(H)$ is the set of elements of $H$, the weight of an edge is its multiplicity in $H$ and $V(H):=V(E(H))$. 
\end{definition}

\begin{definition}
    Let $G=(V,E)$ be a weighted graph. The independent complex $\I_G$ associated to $G$ is a set of multiset subsets $H$ of edges where $H\in \I_G$ if $(V(H),E(H))$ doesn't contain a subgraph which is a cycle of $G$. In addition, the empty set is added to $\I_G$.
\end{definition}

Let $P(G)$ be the power lattice of multiset subsets of the multiset $S$ associated to a graph $G=(V,E)$. It is clear that the independent complex $\I_G$ is a $P(G)$-complex i.e. it is not empty and it is closed under inclusion which we denote by $\preceq$. Now we show that this form a matroid $(P(G),\I_G)$ in the lattice of multiset subsets $P(G)$.

\begin{theorem}
    Let $G=(V,E)$ be a weighted graph. Let $\I_G$ be the independent complex associated to $G$ in the power lattice $P(G)$ associated to $G$. Then, $(P(G),\I_G)$ is a matroid and we call it the graphic matroid associated to a weighted graph $G=(V,E)$.
\end{theorem}
\begin{proof}
    What remains to show is the property (I\ref{I3}) which we recall:

    For all $x,y\in \I_G$ such that $\rho(x)<\rho(y)$, there exists an atom $a\in P$ such that $v_a(x) < v_a(y)$, $x\cup a^{v_a(x)+1} \in \I_G$.

    For the proof, let $x,y\in \I_G$ such that $\rho(x)<\rho(y)$, then there exist an edge $e$ such that $v_e(x)<v_e(y)$. Let $A=\{e\in E\colon v_e(x)<v_e(y)\}$. Therefore $A$ is non-empty. Let $A^C=E\setminus A$ and therefore, for $e\in A^C$, $v_e(x)\geq v_e(y)$.
    We distinguish two cases:
    
    \textbf{Case 1.} Suppose there exists $e\in A$, such that $v_e(x)<w_G(e)-1$. Then $z=x\vee e^{v_e(x)+1} \in \I$, since the only change done to $x$ to obtain $z$ is to increase the weight of one edge but the weight is still small and can't be part of a cycle.
    
    \textbf{Case 2.} Suppose that for all $e\in A$, we have $wt_G(e)-1\leq v_e(x)$. Then $wt_G(e)-1\leq v_e(x)<v_e(y)\leq wt_G(e)$. Hence, for all $e\in A$, $wt_G(e)=v_e(y)=v_e(x)+1$.
    Since $\rho(x)<\rho(y)$, $\sum_{e\in E} v_e(x)<\sum_{e\in E} v_e(y)$ and therefore
    \begin{equation}\label{eq:ex}
    \sum_{e\in A^C} v_e(x)<\sum_{e\in A^C} v_e(y)+|A|.
    \end{equation} 
    Let
    \[
    B=\{e\in A^C\colon v_e(x)=v_e(y)=wt_G(e)\}\]
    and 
    \[
    C=\{e\in A^C\colon v_e(x)>v_e(y)\text{ or } v_e(x)<wt_G(e)\}.
    \]
    Hence $B\cup C=A^C$.
    Equation \eqref{eq:ex} implies
    \[
        \sum_{e\in B} v_e(x)+\sum_{e\in C} v_e(x)<\sum_{e\in B} v_e(y)+\sum_{e\in C} v_e(y)+|A|.
    \]
    By definition of $B$, the previous inequality implies
    \[
        \sum_{e\in C} v_e(x)-\sum_{e\in C} v_e(y)<|A|.
    \]
    Let $D=\{e\in A^C\colon v_e(x)>v_e(y)\}$. We have $D\subseteq C$. Then
    \[
        \sum_{e\in D} v_e(x)-\sum_{e\in D} v_e(y)\leq \sum_{e\in C} v_e(x)-\sum_{e\in C} v_e(y)<|A|.
    \]
    Hence
    \begin{equation}\label{eq:DBAB}
    |D|<|A| \quad \text{i.e. } |D|+|B|<|A|+|B|
    \end{equation}
    Notice that $B\cap D=B\cap A=\emptyset$. Consider the simple graph $G_x=(V,E_x)$ where $e\in E_x$ if $v_e(x)=wt_G(e)$. Similarly, define $G_y=(V,E_y)$ where $e\in E_y$ if $v_e(y)=wt_G(e)$. 
    One can check that $E_x\subseteq B\cup D$ and $A\cup B\subseteq E_y$. These together with Equation \eqref{eq:DBAB} implies $|E_x|<|E_y|$.
    The number of connected components in $G_x$ is $N_x=|V|-|E_x|$ and the number of connected components in $G_y$ is $N_y=|V|-|E_y|$ and we have $N_y<N_x$. By the pigeonhole principle, there is a component $z$ of $G_y$ which contains vertices from two or more components of $G_x$. Take a path $z$ in $G_y$ which connects two components of $G_x$. This path contains an edge $a$ (i.e. $a\preceq z$ which connect two components of $G_x$. Therefore adding this edge to $G_x$ cannot produce a cycle when adding $a$ to the graph $G_x$. More importantly, $x\vee a^{v_a(x)+1}$ does not contain a cycle and hence it is independent.
  
\end{proof}
\section{Stanley-Reisner rings associated to multicomplexes}\label{sec:stanleyreisnerrings}




The face rings or Stanley-Reisner rings associated to shellable simplicial complexes are Cohen-Macaulay rings \cite{Stan}. 
Thus for a shellable $\PP$-complex in a power lattice, the Stanley-Reisner ring associated with its order complex is Cohen-Macaulay.
In this section, we consider a quotient of polynomial ring associated to a multicomplex and show that the ring is sequentially Cohen-Macaulay if the multicomplex is $\PP$-shellable. 
We also explain why the ring we consider corresponding to a multicomplex is a generalization of Stanley-Reisner ring of a simplicial complex. To do this, we adapt the construction of sheaf of rings by Yuzvinsky in \cite{Yuz87} for a multicomplex.




Let $P$ be the lattice of multiset subsets of $\1=\prod_{i=1}^l x_i^{n_i}$ and suppose that $\Delta$ is a multicomplex in $P$. In this section, we also assume that $\1\notin \Delta$. To this multiset, we associate the ring $R=\FF[x_1,\dots,x_l]$, where $\FF$ is any arbitrary field.

Let $X=X(\Delta)$ be the set of all meets of facets of $\Delta$. Consider $X$ as a poset where for $A,B
\in X$, $A\preceq B$ if $B\subseteq A$. The minimal elements of the poset $X$ are therefore the elements of $B$.

Now, we construct a sheaf $\A$ of rings on $X$. For each $\sigma\in X$, let the stalk 
\[
(\A)_\sigma=\FF[x_1,\dots,x_l]/(x_1^{v_{x_1}(\sigma)+1},\dots, x_l^{v_{x_l}(\sigma)+1})
\]
be the quotient ring modulo the ideal $(x_1^{v_{x_1}(\sigma)+1},\dots, x_l^{v_{x_l}(\sigma)+1})$.

Let $\prod_{\sigma\in X} (\A)_\sigma$ be the Cartesian product of the stalks associated to the elements of $X$. For $\sigma\preceq \tau$ in $X$ (i.e., $\tau\subseteq \sigma$), we define the ring homomorphism $\rho_{\tau\sigma}:(\A)_\sigma \rightarrow (\A)_\tau$ as the natural 
projection.

One can verify that these define a sheaf of rings $\A$. 
\begin{definition}
Let $\Delta$ be a multicomplex and let $X=X(\Delta)$ be the set of all meets of facets of $\Delta$. The ring of sections $\Gamma(\A)$ on $X$ is defined by
\begin{equation}\label{eq:sheaf}
\Gamma(\A)=\{ (f_\sigma)_{\sigma\in X} \in \prod_{\sigma\in X} (\A)_\sigma\colon\rho_{\tau\sigma}(f_\sigma)=f_\tau,\sigma\preceq \tau, \sigma,\tau\in X \}.
\end{equation}
\end{definition}

The next theorem shows that the ring of sections on $X$ is, in fact, isomorphic to a quotient of $\FF[x_1, \dots, x_l]$.

\begin{theorem}
Let $P$ be the lattice of multiset subsets of $\hat{1}=\prod_{i=1}^l x_i^{n_i}$ and let $\Delta$ be a multicomplex in $P$. Suppose $X=X(\Delta)$ is the set of all meets of facets of $\Delta$ and $I_\Delta$ is the ideal of $\FF[x_1, \dots, x_l]$ generated by the monomials corresponding to the elements in $P\setminus \Delta$. Then the ring of sections on $X$ is isomorphic to the quotient of the polynomial ring by $I_{\Delta}$, i.e.,
   \[
   \Gamma(\A)\simeq \FF[x_1,\dots,x_l]/I_\Delta.
   \]
\end{theorem}

\begin{proof}
    Consider the homomorphism \begin{align*}
\phi:\FF[x_1,\dots,x_l]&\longrightarrow \Gamma(\A)\\
f&\longmapsto (f+ (x_1^{v_{x_1}(\sigma)+1},\dots ,x_l^{v_{x_l}(\sigma)+1}))_{\sigma\in X}.
\end{align*}
The above map is a well-defined ring homomorphism. We first show that $\ker \phi=I_\Delta$.



    Let $f \in I_\Delta$ be a generator monomial. If $\phi(f) \neq 0$, then there exists $\sigma \in X$ such that $f \notin (x_1^{v_{x_1}(\sigma)+1}, \ldots, x_l^{v_{x_l}(\sigma)+1})$. This implies that $v_{x_i}(f) \leq v_{x_i}(\sigma)$ for all $i \in \{ 1, \ldots, l\}$, which contradicts that $f$ is a monomial corresponding to a nonface, i.e., elements in $P\backslash \Delta$. Thus $f \in \ker \phi$ and therefore, $I_\Delta \subseteq \ker \phi$. 

    Conversely, let $f \in \ker \phi$, i.e., $f \in (x_1^{v_{x_1}(\sigma)+1},\dots ,x_l^{v_{x_l}(\sigma)+1})$ for all $\sigma \in X$. For for any $\sigma \in X$, we associate the ideal $P_{\sigma} = (x_1^{v_{x_1}(\sigma)+1},\dots ,x_l^{v_{x_l}(\sigma)+1})$. Thus $f \in \ker \phi$ implies that $f \in \bigcap\limits_{\sigma \in X} P_{\sigma}$. We claim that $I_\Delta = \bigcap\limits_{\sigma \in X} P_{\sigma}$. That would imply that $f\in I_\Delta$ and therefore $\ker \phi \subseteq I_\Delta$. Thus $\ker \phi= I_\Delta$. Finally, we note that the map $\phi$ is a surjective map follows from Equation \eqref{eq:sheaf} since $\emptyset \in \Delta$ and $(\A)_\emptyset = \mathbb{F}[x_1, \ldots, x_l]$. This would complete the proof of the Theorem.

    \textbf{Proof of the claim:} It is well known that if $J$ is a monomial ideal of $\mathbb{F}[x_1, \ldots, x_l]$, then $(ab, J) = (a, J) \bigcap (b ,J)$ where $a,b$ are relatively prime monomials. 
    Thus the monomial ideal $I_{\Delta}$ can be written as intersection of ideals of the form $(x_1^{b_1}, \ldots, x_l^{b_l})$ i.e.,
    \[
    I_{\Delta} = \bigcap_{b_1,\dots,b_l} (x_1^{b_1}, \ldots, x_l^{b_l}).
    \]
    Note that if $I_\Delta \subseteq ( x_1^{b_1}, \ldots, x_l^{b_l})$, then the multiset $\prod_{i=1}^l x_i^{b_i-1}$ is a face of $\Delta$. 
    Indeed, otherwise, if it is a nonface, the corresponding monomial $\prod_{i=1}^l x_i^{b_i-1} \in I_\Delta \subseteq ( x_1^{b_1}  \ldots, x_l^{b_l})$, which is not possible, considering the degrees. Thus the ideal $I_{\Delta}$ is of the intersection of ideals of the form $P_{\sigma}$ for \emph{some} faces $\sigma \in \Delta$ i.e.,
    \begin{equation}\label{eq:IDPD}
    I_\Delta= \bigcap_{\sigma\in \Gamma} P_\sigma,
    \end{equation}
    where $\Gamma\subseteq \Delta$.
    On the other hand, we claim that for any facet $\sigma$ of $\Delta$, $I_\Delta \subseteq P_{\sigma}$. Let $f = \prod_{i=1}^l x_i^{b_i}$ be any generator of $I_\Delta$ corresponding to a nonface. That means for any facet $\sigma$, $b_i > v_{x_i}(\sigma)$ for some $i$. This implies $\prod_{i=1}^l x_i^{b_i} \in P_{\sigma}$ which proves the claim. Thus we have \begin{equation}\label{eq:faceDelta}
        I_\Delta\subseteq \bigcap\limits_{\sigma\text{ facet of }\Delta}P_\sigma.
    \end{equation}
    Moreover, if $\tau \subset \sigma$, then $P_{\sigma} \subset P_{\tau}$. Since any facet $\sigma$ of  $\Delta$ is an element of $X$, we have the following equalities 
    \begin{equation}\label{3equalities}
    \bigcap_{\sigma\text{ facet of }\Delta}P_\sigma= \bigcap_{\sigma\in X}P_\sigma=\bigcap_{\sigma\in \Delta}P_\sigma 
    \end{equation}
    Now, since $\bigcap\limits_{\sigma\in \Delta}P_\sigma\subseteq \bigcap\limits_{\sigma\in \Gamma}P_\sigma$, combining equations \eqref{eq:IDPD}, \eqref{eq:faceDelta} and \eqref{3equalities}, we get 
 $I_{\Delta} \subseteq \bigcap\limits_{\sigma\in X}P_\sigma = \bigcap\limits_{\sigma\in \Delta}P_\sigma \subseteq \bigcap\limits_{\sigma\in \Gamma} P_\sigma = I_{\Delta}$ and that implies $I_\Delta = \bigcap\limits_{\sigma \in X} P_{\sigma}$. This completes the proof of the claim. 

\end{proof}

In the last theorem, if we consider $\Delta$ to be a simplicial complex, then the ring of sections becomes isomorphic to the Stanley-Reisner ring associated with $\Delta$. This leads to the next definition.

\begin{definition}
Let $P$ be the lattice of multiset subsets of $\hat{1}=\prod\limits_{i=1}^l x_i^{n_i}$ and let $\Delta$ be a multicomplex in $P$. The Stanley-Reisner ideal associated to $\Delta$ is the ideal $I_\Delta$ generated by the monomials corresponding to the elements in $P\setminus \Delta$. The quotient ring $\FF[x_1,\dots,x_n]/I_\Delta$ is called the Stanley-Reisner ring associated with $\Delta$.
\end{definition}
\begin{remark}
  Notice that our ideal $I_\Delta$ is the monomial ideal generated by the monomials in $P\setminus \Delta$. On the other hand, the ideal considered in \cite{HP06,Jah07,Mur11} is the one generated by \emph{any} monomials which are not in $\Delta$, regardless of whether they are in $P$ or not. The difference is demonstrated by an example in \cite[Remark 1.5]{Cim06}.
  \end{remark}
Our goal is to show that the shellability of $\Delta$ implies that $\FF[x_1,\dots,x_l]/I_\Delta$ is sequentially Cohen-Macaulay. To do this, we recall some notions in commutative algebra. For detail, one can refer to, e.g., the `green book' by Stanley \cite{Stan}.

\begin{definition}\label{defn:polarisation}
    Let $R=\FF[x_1,\cdots,x_l]$ be a polynomial ring. Suppose $M_k=x_1^{a_{k,1}}\dots x_l^{a_{k,l}}$ is a monomial in $R$. We define the polarization of $M_k$ to be the square-free monomial
    \[
        pol(M_k)=x_{1,1}x_{1,2}\dots x_{1,a_{k,1}}x_{2,1}x_{2,2}\dots x_{2,a_{k,2}}\dots x_{l,1}x_{l,2}\dots x_{l,a_{k,l}}
    \]
    in the polynomial ring $S=\FF[x_{i,j}\colon 1\leq i\leq l,1\leq j\leq a_{k,i}]$.
    
If $I=(M_1,\dots,M_t)$ is an ideal of $R$, then the polarization of $I$ is the ideal,
\[
pol(I)=(pol(M_1),\dots,pol(M_t)),
\]
in the polynomial ring $\FF[x_{i,j}:1\leq i\leq l,1\leq j\leq \max_k a_{k,i}]$.
\end{definition}

\begin{definition}
Let $R$ be a commutative Noetherian local ring and let $M$ be an $R$-module. 
An element $a\in M$ is said to be $M$-regular if $ax\neq 0$ for all $0\neq x\in M$. A sequence $(a_1,\dots,a_r)$ of elements of $M$ is an $M$-sequence if
\begin{enumerate}[(a)]
    \item $a_1$ is $M$-regular and $a_{i+1}$ is $M/a_i M$-regular for every $i\geq 1$, and
    \item $M/(\sum_i a_i M)\neq 0$.
\end{enumerate}
\end{definition}

\begin{definition}
Let $R$ be a commutative Noetherian local ring and let $M$ be a finitely generated $R$-module. Every maximal $M$-sequences have the same length $\depth(M)$, which is called the depth of $M$.
\end{definition}

\begin{definition}
Let $R$ be a commutative Noetherian local ring and let $M$ be a finitely generated $R$-module. The Krull dimension $\dim M$ of the module $M$ is the Krull dimension of the ring $R/(Ann_R(M))$, where 
$$Ann_R(M)=\{x\in R\colon xy=0 \text{ for all } y\in M\},$$ the annihilator of $M$. 
\end{definition}

\begin{definition}
Let $R$ be a commutative Noetherian local ring and let $M$ be a finitely generated $R$-module. We say that $M$ is a Cohen-Macaulay module if $\depth (M) =\dim M$.
\end{definition}

\begin{definition}
Let $R$ be a commutative Noetherian local ring and let $M$ be a finitely generated (graded) $R$-module. A finite filtration $0=M_0\subset M_1\subset \dots \subset M_l=M$ of $M$ by submodules of $M$ is called a $CM$-filtration if every $M_i/M_{i-1}$ is Cohen-Macaulay and
\[
\dim (M_1/M_0)< \dim (M_2/M_1) < \dots < \dim (M_l/M_{l-1}).
\]
$M$ is said to be sequentially Cohen-Macaulay if $M$ admits a $CM$-filtration.
\end{definition}

For the rest of this paper, we fix an arbitrary multicomplex $\Delta$ in a lattice $P$ of multiset subsets of $\1=\prod\limits_{i=1}^l x_i^{n_i}$ and let $I_\Delta$ be the monomial ideal generated by all the monomials which are in $P\setminus \Delta$ i.e., $I_\Delta=(M_1,\dots,M_{|P\setminus \Delta|})$ where $M_i$'s are the elements of $P\setminus \Delta$. To prove the sequential Cohen-Macaulayness of $R/I_\Delta$ when $\Delta$ is shellable, we will use the following relation between an ideal and its polarization.

\begin{proposition}[{\cite[Proposition 4.11]{Sar05}}]\label{pro:Froberg}
Let $R$ and $S$ be defined as in Definition \ref{defn:polarisation}. Then $R/I_\Delta$ is sequentially Cohen-Macaulay if and only if $S/pol(I_\Delta)$ is sequentially Cohen-Macaulay.
\end{proposition}

A set corresponds to a square-free monomial, when written multiplicatively. We let $P_1$ be the lattice of subsets of $\prod\limits_{ i= 1}^{l } \prod\limits_{ j=1}^{n_i} x_{i,j}$. Let $J$ be the set of all monomials of $S$ which  are in $pol(I_\Delta)$. Then $\Delta_1=P_1\setminus J$ is a $\PP$-complex and it is clearly a simplicial complex. It is also clear by definition that the  Stanley-Reisner ring associated to $\Delta_1$ is given by $S/pol(I_\Delta)$.
The simplicial complex $\Delta_1$ is called the {\em polarized simplicial complex} associated to the multicomplex $\Delta$.

Here is another description of the polarized simplicial complex associated to $\Delta$ that will be useful later.


\begin{theorem}\label{thm:faces}
Let $\Delta_1$ be the polarized simplicial complex associated to a multicomplex $\Delta$. Then
    \[
    \Delta_1=\left\langle \prod\limits_{1 \le i \le l}\prod\limits_{\substack{ 1 \le j \le n_i \\ j\neq a_{i}+1}}x_{i,j}\colon x_1^{a_1}\cdots x_l^{a_l}\in \Delta\right\rangle.
    \]
\end{theorem}
\begin{proof}
    Let $\tilde{\Delta}$ be the simplicial complex given by
    \[
    \tilde{\Delta}=\left\langle \prod\limits_{1\leq i\leq l}\prod\limits_{\substack{ 1\leq j\leq n_i \\ j\neq a_{i}+1}}x_{i,j}\colon x_1^{a_1}\cdots x_l^{a_l}\in \Delta\right\rangle.
    \]
    Our goal is to show that $\Delta_1=\tilde{\Delta}$. 
    
    First, we show that $\tilde{\Delta} \subseteq \Delta_1$. To do this, since $\Delta_1$ is a simplicial complex, it is enough to show that $\prod\limits_{1\leq i\leq l}\prod\limits_{\substack{ 1\leq j\leq n_i \\ j\neq a_{i}+1}}x_{i,j}\in \Delta_1$ for every $x_1^{a_1}\cdots x_l^{a_l}\in \Delta$. So, let $x_1^{a_1}\cdots x_l^{a_l}\in \Delta$ and, by contradiction, suppose instead that \[
    \prod\limits_{1\leq i\leq l}\prod\limits_{\substack{ 1\leq j\leq n_i \\ j\neq a_{i}+1}}x_{i,j}\notin \Delta_1.
    \]
    Hence $\prod\limits_{1\leq i\leq l}\prod\limits_{\substack{ 1\leq j\leq n_i \\ j\neq a_{i}+1}}x_{i,j}\in J$. Since $J$ is a monomial ideal,   $\prod\limits_{1\leq i\leq l}\prod\limits_{\substack{ 1\leq j\leq n_i \\ j\neq a_{i}+1}}x_{i,j}$ is divisible by a monomial which is a generator in $pol(I_\Delta)$. More precisely $\prod\limits_{1\leq i\leq l}\prod\limits_{\substack{ 1\leq j\leq n_i \\ j\neq a_{i}+1}}x_{i,j}$ is divisible by $\prod\limits_{1\leq i\leq l}\prod\limits_{1\leq j\leq b_i}x_{i,j}$, where $b_i\leq a_i$ and $x_1^{b_1}\cdots x_l^{b_l}\notin \Delta$. But $x_1^{a_1}\cdots x_l^{a_l}\in \Delta$ implies that $x_1^{b_1}\cdots x_l^{b_l}\in \Delta$ and hence we have a contradiction. Therefore, by contradiction,  $\prod\limits_{1\leq i\leq l}\prod\limits_{\substack{ 1\leq j\leq n_i \\ j\neq a_{i}+1}}x_{i,j}\in \Delta_1$ and we are done with the first part of this proof.

    Conversely, let $u=\prod\limits_{1\leq i\leq l}\prod\limits_{1\leq j\leq n_i} x_{i,j}^{r_{i,j}}\in \Delta_1$, where $r_{i,j}\in \{0,1\}$. Since $u\in \Delta_1$, then $u\notin pol(I_\Delta)$.
    Let $b_i=\max\{k\colon r_{i,j}=1, \text{ for all } 1\leq j\leq k\}$ (if the set is empty, we take $b_i=0)$. 
    Now, $v=\prod\limits_{1\leq i\leq l}\prod\limits_{1\leq j\leq b_i} x_{i,j}$ is not in $pol(I_\Delta)$ (otherwise $u$ would be in $pol(I_\Delta)$ as $v$ divides $u$. Since $v\notin pol(I_\Delta)$, then $\prod\limits_{i=1}^l x_i^{b_i}\notin I_\Delta$. Therefore $\prod\limits_{i=1}^l x_i^{b_i}\in \Delta$. 
    By definition of $\tilde{\Delta}$, $w=\prod\limits_{1\leq i\leq l}\prod\limits_{\substack{ 1\leq j\leq n_i \\ j\neq b_{i}+1}}x_{i,j}\in \tilde{\Delta}$. 
    The only variables which are not in $w$ are the variables in $\prod\limits_{i=1}^l x_{i,b_{i}+1}$. And by the definition of the $b_i$'s, those are variables neither in $u$.
    Therefore $u$ divides $w$ and since $\tilde{\Delta}$ is a simplicial complex and $w\in \tilde{\Delta}$, we have $u\in \tilde{\Delta}$. Hence $\Delta_1\in \tilde{\Delta}$.


\end{proof}

\begin{remark}
    Notice that the polarized simplicial complex $\Delta_1$ associated with $\Delta$ is not necessarily pure. For example, consider the power lattice of the multiset subsets of $ \hat{1} = x_1^3 x_2^3$. Let $\Delta$ be the multicomplex with only two facets $F_1 = x_1^2 x_2^2$ and $F_2 = x_1 x_2^3$. Then the polarized simplicial complex 
    $\Delta_1$ associated to $\Delta$ has at least two facets $x_{1,1}x_{1,2} x_{2,1}x_{2,2}$ and $x_{1,1}x_{1,3} x_{2,1} x_{2,2} x_{2,3}$. 
\end{remark}


\begin{corollary}\label{cor:maxelement}
Let $\Delta_1$ be the polarised simplicial complex associated to the multicomplex $\Delta$. Then   $\Delta_1=\langle F\rangle$ where
    \[
    F=\left\lbrace \prod\limits_{1\leq i\leq l}\prod_{\substack{ 1\leq j\leq n_i \\ j\neq b_{i}+1}} x_{i,j}  \colon 0\leq b_i\leq a_i, \prod_{i=1}^l x_i^{a_i} \text{ is a facet of } \Delta \right\rbrace
    \]
\end{corollary}
\begin{proof}         
This is clear.
\end{proof}

We recall the notion of shellability for non-pure simplicial complexes introduced by Bj{\"o}rner and Wachs.
\begin{definition}
Let $\Delta_1$ be a (not necessarily pure) simplicial complex. $\Delta_1$ is said to be shellable if there exists an ordering of the facets of $\Delta_1$, $F_1,\dots,F_l$ such that for every $i<j$, there is $k$ such that $F_i\cap F_j\subseteq F_k\cap F_j$ and $|F_k\cap F_j| = |F_j|-1$.
\end{definition}

\begin{theorem}\label{thm:shellable}
    If $\Delta$ is a shellable multicomplex, then the polarized simplicial complex $\Delta_1$ associated to $\Delta$, is a (not necessarily pure) shellable simplicial complex.
\end{theorem}
\begin{proof}
    Let us denote by $\prec_\Delta$ the ordering of the facet of $\Delta$ which makes it shellable.
    For this proof, we need to define an ordering on the facets of $\Delta_1$. Notice that an element of $F$ in Corollary \ref{cor:maxelement} is not necessarily a facet of $\Delta_1$. However, a facet of $\Delta_1$ must be an element of $F$.
    Let $u = \prod\limits_{i=1}^l (\prod\limits_{j=1}^{n_i} x_{i,j}\setminus x_{i,b_i+1})$ be a facet of $\Delta_1$. To $u$, we associated the face $U_1=\prod\limits_{i=1}^l x_i^{b_i}$ of $\Delta$ and the facet $U=\prod\limits_{i=1}^l x_i^{a_i}$ such that if $U_1\subseteq U'$ for another facet $U'$ of $\Delta$, then $U\prec_\Delta U'$. This means that $U$ is the smallest facet of $\Delta$ which contains $U_1$.
    Similarly, let $v = \prod\limits_{i=1}^l (\prod\limits_{j=1}^{n_i} x_{i,j}\setminus x_{i,b'_i+1})$, $V_1=\prod\limits_{i=1}^l x_i^{b'_i}$ be the associated face of $\Delta$ and $V=\prod\limits_{i=1}^l x_i^{a'_i}$ be the associated facet. We say that $u\preceq_{\Delta_1} v$ if either $u=v$ or  $u\neq v$ and one of the following is satisfied:
    \begin{enumerate}[(i)]
        \item $U=V$ and if $e$ is the smallest integer such that $b_e\neq b'_e$, then $b_e>b'_e$. 
        \item $U\prec_{\Delta} V$.
    \end{enumerate}
    This is a total ordering $\preceq_{\Delta_1}$ on the facets of $\Delta_1$. We now show that this defines a shelling on $\Delta_1$.
    
    So, let us write $u = \prod\limits_{1\leq i\leq l}\prod\limits_{\substack{ 1\leq j\leq n_i \\ j\neq b_{i}+1}} x_{i,j}$ and $v = \prod\limits_{1\leq i\leq l}\prod\limits_{\substack{ 1\leq j\leq n_i \\ j\neq b'_{i}+1}} x_{i,j}$ and suppose $u\prec_{\Delta_1} v$. Assume that $u,v$ correspond, respectively, to $U_1=\prod\limits_{i=1}^l x_i^{b_i}\subseteq \prod\limits_{i=1}^l x_i^{a_i}=U$ and $V_1=\prod\limits_{i=1}^l x_i^{b'_i}\subseteq \prod\limits_{i=1}^l x_i^{a'_i}=V$.
    There are two cases:\
    
        \textbf{Case 1.} $U=V$. Let $e$ be the smallest integer such that $b_e\neq b'_e$. In this case $b_e>b'_e$. 
        We can then take $W_1 = x_e^{b_e}\prod\limits_{\substack{1 \le i \le l \\i\neq e}} x_i^{b'_i}\subseteq U$. Since $V_1\subseteq W_1$, then $U$ must also be the smallest (w.r.t $\preceq_\Delta$) which contains $W_1$. 
        This defines $w=\prod\limits_{1\leq i\leq l}\prod\limits_{\substack{ 1\leq j\leq n_i \\ j\neq b^*_{i}+1}} x_{i,j}$, where $b^*_i = b'_i$ for $i\neq e$ and $b^*_e=b_e$. It is not hard to see that $w\prec_{\Delta_1} v$. What remains is to show that $u\cap v\subseteq w\cap v$ and $|w\cap v| = |v|-1$. 
        \begin{align*}
            u\cap v &= \left(\prod\limits_{1\le i \le l}\prod_{\substack{ j=1 \\ j\neq b_{i}+1}}^{n_i} x_{i,j} \right) \cap \left(\prod_{1\leq i\leq l}\prod_{\substack{ 1\leq j\leq n_i \\ j\neq b'_{i}+1}} x_{i,j}\right)\\
            &= \prod\limits_{1\leq i\leq l} \left(\prod_{\substack{ 1\leq j\leq n_i \\ j\neq b_{i}+1}} x_{i,j} \right)\cap \left(\prod_{\substack{ 1\leq j\leq n_i \\ j\neq b'_{i}+1}} x_{i,j} \right)\\
            &= \left(\prod_{\substack{ 1\leq j\leq n_e \\ j\neq b_{e}+1}} x_{e,j} \right)\cap \left(\prod_{\substack{ 1\leq j\leq n_e \\ j\neq b'_{e}+1}} x_{e,j} \right)\cdot \prod\limits_{\substack{1\leq i\leq l\\i\neq e}} \left(\prod_{\substack{ 1\leq j\leq n_i \\ j\neq b_{i}+1}} x_{i,j} \right)\cap \left(\prod_{\substack{ 1\leq j\leq n_i \\ j\neq b'_{i}+1}} x_{i,j} \right)
        \end{align*}
        But for $i\neq e$,
        \[
        \left(\prod_{\substack{ 1\leq j\leq n_i \\ j\neq b_{i}+1}} x_{i,j} \right)\cap \left(\prod_{\substack{ 1\leq j\leq n_i \\ j\neq b'_{i}+1}} x_{i,j} \right) \subseteq \left(\prod_{\substack{ 1\leq j\leq n_i \\ j\neq b'_{i}+1}} x_{i,j} \right) = \left(\prod_{\substack{ 1\leq j\leq n_i \\ j\neq b^*_{i}+1}} x_{i,j} \right)\cap \left(\prod_{\substack{ 1\leq j\leq n_i \\ j\neq b'_{i}+1}} x_{i,j} \right).
        \]
        Hence 
        \begin{align*}
            u\cap v &\subseteq \left(\prod_{\substack{ 1\leq j\leq n_e \\ j\neq b_{e}+1}} x_{e,j} \right)\cap \left(\prod_{\substack{ 1\leq j\leq n_e \\ j\neq b'_{e}+1}} x_{e,j} \right)\cdot \prod\limits_{\substack{1\leq i\leq l\\i\neq e}} \left(\prod_{\substack{ 1\leq j\leq n_i \\ j\neq b^*_{i}+1}} x_{i,j} \right)\cap \left(\prod_{\substack{ 1\leq j\leq n_i \\ j\neq b'_{i}+1}} x_{i,j} \right)
        \end{align*}
        Since $b_e=b^*_e$, then
        \begin{align*}
            u\cap v &\subseteq \prod\limits_{1\leq i\leq l} \left(\prod_{\substack{ 1\leq j\leq n_i \\ j\neq b^*_{i}+1}} x_{i,j} \right)\cap \left(\prod_{\substack{ 1\leq j\leq n_i \\ j\neq b'_{i}+1}} x_{i,j} \right)\\
             &\subseteq  \left(\prod\limits_{1\leq i\leq l}\prod_{\substack{ 1\leq j\leq n_i \\ j\neq b^*_{i}+1}} x_{i,j} \right)\cap \left(\prod\limits_{1\leq i\leq l}\prod_{\substack{ 1\leq j\leq n_i \\ j\neq b'_{i}+1}} x_{i,j} \right)\\
            &\subseteq w\cap v.
        \end{align*}
        Finally, 
        \begin{align*}
            |w\cap v| &= \left| \left(\prod\limits_{1\leq i\leq l}\prod_{\substack{ 1\leq j\leq n_i \\ j\neq b^*_{i}+1}} x_{i,j} \right)\cap \left(\prod\limits_{1\leq i\leq l}\prod_{\substack{ 1\leq j\leq n_i \\ j\neq b'_{i}+1}} x_{i,j} \right) \right| \\
            &= \left| \prod\limits_{1\leq i\leq l} \left(\prod_{\substack{ 1\leq j\leq n_i \\ j\neq b^*_{i}+1}} x_{i,j} \right)\cap \left(\prod_{\substack{ 1\leq j\leq n_i \\ j\neq b'_{i}+1}} x_{i,j} \right) \right|\\
            &= \left|\left(\prod_{\substack{ 1\leq j\leq n_e \\ j\neq b_{e}+1}} x_{e,j} \right)\cap \left(\prod_{\substack{ 1\leq j\leq n_e \\ j\neq b'_{e}+1}} x_{e,j} \right) \cdot \prod\limits_{\substack{1\leq i\leq l \\ i\neq e}}  \left(\prod_{\substack{ 1\leq j\leq n_i \\ j\neq b'_{i}+1}} x_{i,j} \right)\right|.
        \end{align*}
        But \[
            \left|\left(\prod_{\substack{ 1\leq j\leq n_e \\ j\neq b_{e}+1}} x_{e,j} \right)\cap \left(\prod_{\substack{ 1\leq j\leq n_e \\ j\neq b'_{e}+1}} x_{e,j} \right)\right| = \left|\left(\prod_{\substack{ 1\leq j\leq n_e \\ j\neq b'_{e}+1}} x_{e,j} \right)\right|-1.
        \]
        Hence $|w\cap v| = |v|-1$
        
         \textbf{Case 2.} $U\prec V$. By shellability of the multicomplex, there is $W\prec_\Delta V$ such that $U\cap V\subseteq W\cap V$ and $|W\cap V|=|V|-1$. For $|W\cap V|=|V|-1$ to hold, $W$ must be of the form $W=x_m^{a'_m+1}x_n^{a'_n-1}\prod\limits_{\substack{1 \le i \le l \\i\neq m,n}} x_i^{a'_i}$ for some $1\leq m,n\leq l$. Since $W\prec_{\Delta}V$ and $V$ being the ``smallest'' to contain $V_1$, we must have $V_1\nsubseteq W$. Therefore $b'_n\geq a'_n$. By looking at the valuation at $x_n$, $v_{x_n}(W\cap V)\leq v_{x_n}(W)=a'_n-1$. Furthermore $U\cap V\subseteq W\cap V$, therefore $v_{x_n}(U\cap V) \leq v_{x_n}(W\cap V)\leq a'_n-1$. Since $v_{x_n}(V) = a'_n$ and $v_{x_n}(U\cap V) = \min\{ v_{x_n}(U),v_{x_n}(V) \}$, we must have $a_n = v_{x_n}(U)\leq a'_n-1$. Therefore $b_n\leq a_n\leq a'_n-1$. So we can take $W_1 = x_n^{b_n}\prod\limits_{\substack{1\le i\le l \\ i\neq n}} x_i^{b'_i}$ and we have $W_1\subseteq W$. This gives us $w=\prod\limits_{1\leq i\leq l}\prod\limits_{\substack{ 1\leq j\leq n_i \\ j\neq b^*_{i}+1}} x_{i,j}$, where $b^*_i=b'_i$ for $i\neq n$ and $b^*_n=b_n$. Since $W\prec_\Delta V$, then $w\prec_{\Delta_1} v$. Now $b_n\leq a'_n-1<a'_n\leq b'_n$ so that $b^*_n\neq b'_n$ and similarly to the previous case, one can show that $u\cap v \subseteq w\cap v$ and $|w\cap v|=|v|-1$.

\end{proof}

We record here a result about sequential Cohen-Macaulayness of the Stanley-Reisner ring associated to a non-pure shellable simplicial complex.

\begin{proposition}[{\cite[Sec III.2]{Stan}}]\label{pro:polCM}
    If $\Delta_1$ is a (not necessarily pure) shellable simplicial complex whose associated Stanley-Reisner ideal in a polynomial ring $R_1$ is $I_{\Delta_1}$, then $R_1/I_{\Delta_1}$ is sequentially Cohen-Macaulay.
\end{proposition}

As a consequence of Proposition \ref{pro:polCM}, Theorem \ref{thm:shellable} and Proposition \ref{pro:Froberg}, we have the main theorem of this section.

\begin{theorem}
    If $\Delta$ is a shellable multicomplex in the power lattice of multiset subsets of $x_1^{n_1}\cdots x_l^{n_l}$, then the Stanley-Reisner ring $\FF[x_1,\dots,x_l]/I_\Delta$ is a sequentially Cohen-Macaulay ring.
\end{theorem}

\paragraph{\textbf{Acknowledgement}} We thank Ananthnarayan Hariharan and Neeraj Kumar for helpful discussion on questions related to Stanley-Reisner rings.

\bibliographystyle{plain}
\bibliography{references}

\end{document}